\documentclass[12pt]{amsart}

\usepackage[all]{xy}
\usepackage{url}
\usepackage{am ssymb}
\usepackage{hyperref}
\usepackage[margin=1 in]{geometry}
\usepackage{color,soul}
\theoremstyle{plain}
\newtheorem{theorem}[subsubsection]{Theorem}
\newtheorem{proposition}[subsubsection]{Proposition}
\newtheorem{lemma}[subsubsection]{Lemma}

\newtheorem{corollary}[subsubsection]{Corollary}

\theoremstyle{definition}
\newtheorem{definition}[subsubsection]{Definition}
\newtheorem{example}[subsubsection]{Example}

\newtheorem{remark}[subsubsection]{Remark}

\newtheorem{Theorem}[subsection]{Theorem}
\newtheorem{Proposition}[subsection]{Proposition}
\newtheorem{Lemma}[subsection]{Lemma}
\newtheorem{Corollary}[subsection]{Corollary}

\theoremstyle{definition}

\newtheorem{Remark}[subsection]{Remark}

\numberwithin{equation}{subsubsection}
\numberwithin{equation}{subsubsection}
\normalfont
\usepackage[T1]{fontenc}{}
\usepackage{setspace}

\usepackage{tikz-cd}

\makeatletter
\def\@tocline#1#2#3#4#5#6#7{\relax
  \ifnum #1>\c@tocdepth 
  \else
    \par \addpenalty\@secpenalty\addvspace{#2}%
    \begingroup \hyphenpenalty\@M
    \@ifempty{#4}{%
      \@tempdima\csname r@tocindent\number#1\endcsname\relax
    }{%
      \@tempdima#4\relax
    }%
    \parindent\z@ \leftskip#3\relax \advance\leftskip\@tempdima\relax
    \rightskip\@pnumwidth plus4em \parfillskip-\@pnumwidth
    #5\leavevmode\hskip-\@tempdima
      \ifcase #1
       \or\or \hskip 1em \or \hskip 2em \else \hskip 3em \fi%
      #6\nobreak\relax
    \dotfill\hbox to\@pnumwidth{\@tocpagenum{#7}}\par
    \nobreak
    \endgroup
  \fi}
\makeatother
\usepackage{hyperref}

\title{Frobenius and Valuation Rings}
\author{Rankeya Datta and Karen E. Smith}
\email{rankeya@umich.edu, kesmith@umich.edu}
\thanks{The first author was partially supported by the Juha Heinonen Memorial Graduate Fellowship at the University of Michigan. The second author was partially supported by NSF grants DMS-1001764 and DMS-1501625.}

\begin{document}

\maketitle

\begin{abstract} The behavior of the Frobenius map is investigated for valuation rings of prime characteristic. We show that valuation rings are always F-pure. We introduce a generalization of the notion of strong F-regularity, which we call F-pure regularity, and show that a valuation ring is F-pure regular if and only if it is Noetherian. For valuations on function fields, we show that 
the Frobenius map is finite if and only if the valuation is Abhyankar; in this case the valuation ring is Frobenius split. 
For Noetherian valuation rings in function fields, we show that the valuation ring is Frobenius split if and only if Frobenius is finite, or equivalently, if and only if the valuation ring is excellent. 
\end{abstract}

\renewcommand{\baselinestretch}{0.75}\normalsize
\tableofcontents
\renewcommand{\baselinestretch}{1.0}\normalsize

\addtocontents{toc}{~\hfill\textbf{Page}\par}


\section{Introduction}

Classes of singularities defined using Frobenius---F-purity, Frobenius splitting, and the various variants of F-regularity---have  played a central role in commutative algebra and algebraic geometry over the past forty years.  
The goal of this paper is a systematic study of these {F-singularities} in the novel, but increasingly important {\it non-Noetherian} setting of {\bf valuation rings.}

{}

Let $R$ be a  commutative ring of prime characteristic $p$. The Frobenius map is the ring homomorphism $R \overset{F}\to R $ sending each element to its $p$-th power. While simple enough, the Frobenius map reveals deep structural properties of a Noetherian  ring of prime characteristic, and is a powerful tool for proving theorems for rings containing an arbitrary  field (or varieties, say,  over $\mathbb C$)  by standard reduction to characteristic $p$ techniques. Theories such as  Frobenius splitting \cite{MR} and tight closure \cite{HH2} are well-developed in the Noetherian setting, often under the additional assumption that the Frobenius map is finite. Since classically most motivating problems were inspired by  algebraic geometry and representation theory, these assumptions seemed  natural and not very restrictive. 
Now, however,  good reasons are emerging to study F-singularities in certain  non-Noetherian settings  as well.

{}

One such setting is cluster algebras \cite{FoZe}. An upper cluster algebra over $\mathbb F_p$ need not be Noetherian, but  recently it was shown that it is always Frobenius split, and indeed, admits a ``cluster canonical'' Frobenius splitting \cite{BeMuRaSm}.
Likewise  valuation rings are  enjoying a resurgence of popularity despite rarely being Noetherian, with  renewed interest in non-Archimedean geometry \cite{Conrad},  the development of  tropical  geometry \cite{GRW}, and the valuative tree \cite{FavJon}, to name just a few examples, as well as fresh uses in higher dimensional birational geometry (e.g. \cite{Cut, BobPe, Bouck}).  

{}

For a Noetherian ring $R$, the Frobenius map is flat if and only if $R$ is regular, by a famous theorem of Kunz \cite{Kunz}. As we observe in Theorem \ref{flatness and faith flatness of Frobenius for Valuation rings}, the Frobenius map is always flat for a valuation ring. So in some sense, a valuation ring of characteristic $p$ might be interpreted as a ``non-Noetherian regular ring.''

{}

On the other hand, some valuation rings are decidedly more like the local rings of smooth points on varieties than others. For example, for a  variety $X$ (over, say, an algebraically closed field of characteristic $p$), the  Frobenius map  is always finite. For 
 valuation rings of the function field of $X$, however,  we show that the Frobenius is finite if and only the valuation is Abhyankar; see
     Theorem \ref{F-finiteness equivalent to Abhyankar}.  In particular, for discrete valuations, finiteness of Frobenius is equivalent to the valuation being {\it divisorial}---that is, given by  the order of vanishing along a prime divisor on some birational model. 
Abhyankar valuations might be considered the geometrically most interesting ones (C.f. \cite{SLE}), so it is fitting that their valuation rings behave the most like the rings of smooth points on a variety.
 Indeed, recently, the local uniformization problem for Abhyankar valuations was settled in positive  characteristic \cite{Knaf-Kuhl}.

{}

One can weaken the demand that Frobenius is flat and instead require only that the Frobenius map is {\it pure} (see \S \ref{F-purity and Frobenius splitting}).  Hochster and Roberts observed that this condition, which they dubbed {\it F-purity}, is  often sufficient for controlling singularities of a Noetherian local ring, an observation at the heart of  their famous  theorem on the Cohen-Macaulayness of invariant rings \cite{HR1, HR2}.  We show in Corollary \ref{Frobenius is pure} that any valuation ring of characteristic $p$ is F-pure. 
Purity of a map is equivalent to its splitting under suitable finiteness hypotheses, but at least for valuation rings (which rarely satisfy said hypotheses), 
the purity of Frobenius seems to be better behaved and more straightforward than its splitting. Example \ref{A discrete valuation ring which is not F-finite} shows that not all valuation rings are Frobenius split, even in the Noetherian case.

{}

Frobenius splitting   has well known deep local and  global  consequences for algebraic varieties.  In the local case, Frobenius splitting has been said to be  a ``characteristic $p$ analog''  of  log canonical singularities for complex varieties, whereas related properties  correspond to other singularities in the minimal model program \cite{Hara-Wata, Schw3, Smi4,  Tagaki2}. For projective varieties, Frobenius splitting is  related to positivity of the anticanonical bundle; see  \cite{BrKu, MR, Smi2, SmSch}.   
Although valuation rings are always F-pure, the question of their
Frobenius splitting is subtle.
Abhyankar valuations in function fields  are Frobenius split (Theorem \ref{F-finiteness equivalent to Abhyankar}),  but  a discrete valuation ring is Frobenius split if and only if it is {\it excellent} in the sense of Grothendieck  (Theorem \ref{DVRFsplitChar}). 
Along the way, we prove a simple characterization of the finiteness of Frobenius for a Noetherian domain in terms of excellence, which  gives a large class of Noetherian domains in which Frobenius splitting implies excellence; see \S \ref{F-finiteExcellence} for details.

{}

Closely related to F-purity and Frobenius splitting are the various variants of F-regularity. Strong F-regularity was introduced by Hochster and Huneke \cite{HH1} as a proxy for weak F-regularity---the property that all ideals are tightly closed---because it is easily shown to pass to localizations. Whether or not  a weakly F-regular ring remains so after localization is a long standing open question in tight closure theory, as is the equivalence of  weak F-regularity and strong F-regularity.  Strong F-regularity has found many applications beyond tight closure, and is closely related to Ramanathan's notion of  ``Frobenius split along a divisor"  \cite{Ram, Smi2}. A smattering of applications might include \cite{ AbLeus, BeMuRaSm, Blickle, BrKu,  GLPSTZ,  hacon,  Patak1,  ST1,  Schw2,  SmSch, SmVan,  Smith-Zhang, Smi2}.

{}

Traditionally, strong F-regularity has been defined only for Noetherian  rings in which Frobenius is finite. 
To clarify the situation for valuation rings, we introduce a new definition which we call {\it  F-pure regularity} (see Definition \ref{pure element and pure F-regularity})  requiring {\it purity} rather than {\it splitting} of certain maps. We show that F-pure regularity is better suited for arbitrary rings, but equivalent to strong F-regularity under the standard finiteness hypotheses; it also agrees with  	
another  generalization of strong F-regularity proposed by Hochster (using tight closure) in the local  Noetherian case \cite{Hoc2}.   Likewise, we show  that F-pure regularity is a  natural and straightforward generalization of strong F-regularity, satisfying many expected properties---for example, regular rings are F-pure regular. Returning to valuation rings, in  Theorem \ref{purely F-regular valuation rings} we
  characterize  F-pure regular valuation rings as precisely those that are  Noetherian. 
  
 Finally, in \S \ref{split F-regularity for valuations}, we compare our generalization of strong F-regularity with the  obvious competing generalization, in which the standard definition in terms of splitting certain maps is naively extended without assuming any finiteness conditions.  To avoid confusion,{\footnote{An earlier version of this paper used the  terms {\it pure F-regularity} and {\it split F-regularity} for the two generalizations of classical strong F-regularity, depending upon whether maps were required to be pure or split.  The names were changed at Karl Schwede's urging to avoid confusion with terminology for pairs in \cite{Takagi1}.}}  
 we call this {\it{split F-regularity}}. We  characterize split F-regular valuation rings (at least in a certain large class of fields) as precisely those that are Frobenius split, or equivalently {\it excellent; }  see Corollary \ref{splitDVR}.
 But we also point out  that there are regular local rings that fail to be split F-regular, so perhaps split F-regularity is not a reasonable notion of ``singularity.''

{}

The authors gratefully acknowledge  the conversation with  Karl Schwede and Zsolt Patakfalvi that inspired  our  investigation of Frobenius in  valuation rings, and especially Karl, with whom we later had many  fruitful discussions.
The authors also acknowledge helpful discussions with Mel Hochster,  Daniel Hernandez,  Mattias Jonsson,  Eric Canton,  Linquan Ma, and Juan Felipe P\'erez, some of which occurred  at the  American Mathematical Society's  2015 Mathematics Research Community in commutative algebra, attended by the first author. In addition, Karl and Linquan made detailed comments on an earlier version, which greatly improved the paper. In particular, Linquan  suggested Corollary \ref{Linquan}. We also thank the referee for helpful suggestions.

\section{Preliminaries}

Throughout this paper,  all rings are assumed to be commutative, and  of prime characteristic $p$ unless explicitly stated otherwise.
  By a local ring, we mean a ring with a unique maximal ideal, {\it not necessarily Noetherian.}

\subsection{Valuation Rings}

We recall some basic facts and definitions about valuation rings (of arbitrary characteristic), while fixing notation. See \cite[Chapter VI]{Bourbaki} or \cite[Chapter 4]{Mat} for proofs and details. 

{}

  The symbol $\Gamma$ denotes an ordered abelian group. Recall that such an abelian group is torsion free. The rational rank of $\Gamma$, denoted  $\operatorname{rat.rank} \hspace{1mm} \Gamma$, is the dimension of the $\mathbb Q$-vector space $\mathbb Q \otimes_{\mathbb Z} \Gamma$.

{}
 
  Let $K$ be a field.  A \textbf{valuation on $K$} is a homomorphism 
$$v: K^{\times} \rightarrow \Gamma$$ from the group of units of $K$, satisfying 
 $$v(x + y) \geq \min\{v(x), v(y)\}$$ for all $x, y \in K^{\times}$.  
We say that $v$ \textbf{is defined over a subfield $k$ of $K$,} or that $v$ is a  \textbf{valuation on $K/k$,} if $v$ takes the value $0$ on elements of $k$. 
{}

  There is no loss of generality in assuming that $v$ is surjective, in which case we say $\Gamma$ (or $\Gamma_v$) is the \textbf{value group}  of $v$.
Two  valuations $v_1$ and $v_2$ on $K$ are said to be \textbf{equivalent} if  there is an order preserving isomorphism of their value groups identifying $v_1(x)$ and $v_2(x)$ for all $x\in K^{\times}.$ Throughout this paper, we identify equivalent  valuations. 

{}

  The \textbf{valuation ring of $v$} is the subring $R_v \subseteq K$ consisting of all elements $x\in K^{\times}$ such that $v(x)\geq 0$ (together with the zero element of $K$). 
Two valuations on a field $K$  are {equivalent} if and only if they determine the same valuation ring. Hence a valuation ring of $K$  is essentially the same thing as an equivalence class of valuations on $K$.

{}

The valuation ring of $v$ is local, with {\bf maximal ideal $m_v$} consisting of elements of strictly positive values (and zero).  The {\bf residue  field $R_v/m_v$} is denoted $\kappa(v)$. If $v$ is a valuation over $k$, then both $R_v$ and $\kappa(v)$ are $k$-algebras.

 {}
 
  A valuation ring $V$ of $K$ can be characterized directly, without reference to a valuation, as a subring with the property that for every $x \in K$,  either $x \in V$ or $x^{-1} \in V$. The valuation ring $V$  uniquely determines a valuation $v$ on $K$ (up to equivalence), whose valuation ring  in turn recovers $V$.
 Indeed, it is easy to see that the set of ideals of a valuation ring is totally ordered by inclusion, so  the set of principal ideals $\Gamma^+$ forms a monoid under multiplication, ordered by  $(f) \leq (g)$ whenever $f$ {\it divides} $g$. Thus, $\Gamma$ can be taken to be the ordered abelian group generated by the principal ideals, and the valuation  $v: K^{\times} \rightarrow \Gamma$ is  induced by the monoid map sending  each non-zero $x \in V$ to  the ideal generated by $x$. Clearly, the valuation ring of $v$ is $V$. 
See \cite[Chapter 4]{Mat}.

\subsection{Extension of Valuations}
\label{extension}
Consider an extension of fields  $K \subseteq L$. By definition, a valuation $w$ on $L$ is an  \textbf{extension} of a valuation $v$ on $K$ if the restriction of $w$ to the subfield $K$ is $v$. Equivalently, $w$ extends $v$ if \textbf{$R_w$ dominates $R_v$}, meaning that $R_v = R_w\cap K$ with $m_w \cap R_v = m_v$. In this case, there is an induced map of residue fields
$$\kappa(v) \hookrightarrow \kappa(w).$$ 
 
{}
 
  The \textbf{residue degree of $w$ over $v$}, denoted by $f(w/v)$, is the degree of the residue field extension $\kappa(v)\hookrightarrow \kappa(w).$

{}

  If $w$ extends $v$,  there is a natural injection of ordered groups $\Gamma_v \hookrightarrow \Gamma_w$, since $\Gamma_v$ is the image of $w$ restricted to the subset $K$.
The \textbf {ramification index of $w$ over $v$}, denoted by $e(w/v)$, is the index  of  $\Gamma_v$ in   $\Gamma_w$.

{}

If $K \hookrightarrow L$ is a finite extension, then both the ramification index $e(w/v)$ and the residue  degree $f(w/v)$ are {\it finite}.  Indeed, if $K \subseteq L$ is a degree $n$ extension, then 
\begin{equation}\label{ramification/residue for finite extensions is finite} 
e(w/v)f(w/v) \leq n.
\end{equation}

{}

More precisely, we have
\begin{proposition}\cite[VI.8]{Bourbaki}.
\label{key theorem for F-finiteness results}
Let $K \subseteq L$ be an extension of fields of finite degree $n$. For a valuation $v$ on $K$, consider the set $\mathcal S$ of all  extensions (up to equivalence) $w$ of $v$ to $L$.  
 Then 
 $$\sum_{w_i \in \mathcal S} e(w_i/v)f(w_i/v) \leq  n.
 $$ In particular, the set $\mathcal S$ is finite.
Furthermore, equality holds if and only if  the integral closure of $R_v$ in $L$ is a finitely generated $R_v$-module.
\end{proposition}

\subsection{Abhyankar Valuations}
Fix a field $K$ finitely generated  over a fixed ground field $k$, and let $v$ be a valuation on $K/k$.
By definition, the \textbf{transcendence degree} of $v$ is the transcendence degree of the field extension
$$k \hookrightarrow \kappa(v).$$

{}

  The main result about the transcendence degree of valuations is due to Abhyakar  \cite{Abhyankar}. See also \cite[VI.10.3, Corollary 1]{Bourbaki}.

{}

\begin{theorem} [Abhyankar's Inequality]
\label{abhyankar's theorem} 
Let $K$ be a finitely generated field extension of $k$, and let $v$ be a valuation on $K/k$. Then
\begin{equation}
\operatorname{trans.deg} \hspace{1mm} v + \operatorname{rat.rank} \hspace{1mm} \Gamma_v \leq \operatorname{trans. deg} \hspace{1mm} K/k. \label{Abineq}
\end{equation}

  Moreover if equality holds, then $\Gamma_v$ is a finitely generated abelian group, and $\kappa(v)$ is a finitely generated extension of $k$.
\end{theorem}

{}

  We say $v$ is an \textbf{Abhyankar valuation} if equality holds in Abhyankar's Inequality (\ref{Abineq}). Note that an Abyhankar valuation has a finitely generated value group, and its residue field is finitely generated over the ground field $k$. 

{}

\begin{example}
\label{AbEx}
Let $K/k$ be the function field of a normal algebraic variety $X$ of dimension  $n$ over a ground field $k$. For a  prime divisor $Y$ of $X$, consider the local ring $\mathcal O_{X,Y}$ of  rational functions on $X$ regular at $Y$.  The ring $\mathcal O_{X,Y}$ is a discrete valuation ring, corresponding to a valuation $v$ (the order of vanishing along $Y$) on $K/k$; this valuation is  of rational rank one and  transcendence degree $n-1$ over $k$, hence Abhyankar.  Such a valuation is called a {\textbf {divisorial valuation.}} 
Conversely,  every rational rank one  Abhyankar valuation is  divisorial:
 for such a $v$, there exists some normal model $X$ of $K/k$ and a divisor $Y$ such that $v$ is the order of vanishing along $Y$
 \cite[VI, \S 14, Thm 31]{Z-S}.
\end{example}

{}

\begin{proposition}
\label{extension and restriction of Abhyankar valuations}
Let $K \subseteq L$ be a finite extension of finitely generated field extensions of $k$, and suppose that $w$ is valuation on $L/k$ extending a valuation $v$ on $K/k$. Then $w$ is Abhyankar  if and only if $v$ is Abhyankar.
\end{proposition}

\begin{proof}
Since $L/K$ is  finite,  $L$ and $K$ have the same transcendence degree over $k$. On the other hand, the extension
$ \kappa(v) \subseteq \kappa(w)$ is also finite by (\ref{ramification/residue for finite extensions is finite}), and so $ \kappa(v)$ and $\kappa(w)$ also have the same transcendence degree over $k$.
Again by (\ref{ramification/residue for finite extensions is finite}), since  $\Gamma_w/\Gamma_v$ is a finite abelian group, $\mathbb{Q} \otimes_{\mathbb{Z}} \Gamma_w/\Gamma_v = 0$. By exactness of 
$$0 \rightarrow \mathbb{Q} \otimes_{\mathbb{Z}} \Gamma_v \rightarrow \mathbb{Q} \otimes_{\mathbb{Z}} \Gamma_w \rightarrow \mathbb{Q} \otimes_{\mathbb{Z}} \Gamma_w/\Gamma_z \rightarrow 0$$ 
we conclude that $\Gamma_w $ and $ \Gamma_v$ have the same rational rank.
The result is now clear from the definition of an Abhyankar valuation.
\end{proof}

\subsection{Frobenius}
\label{Frobenius} 
Let $R$ be a ring of prime characteristic $p$. The Frobenius map $R\overset{F}\to R$ is defined by $F(x) = x^p$. We can denote the target copy of $R$ by $F_*R$ 
and view it as an $R$-module via restriction of scalars by $F$; thus $F_*R$ is both a ring (indeed, it is precisely $R$) and an $R$-module in which the action of  $r \in R$  on $x \in F_*R$ produces $ r^px$.  With this notation, the Frobenius map $F:R\rightarrow F_*R$ and its iterates 
 $F^e: R \rightarrow F_*^eR$ are ring maps, as well as $R$-module maps. 
See \cite[1.0.1]{Smith-Zhang} for a  further discussion of this notation. 

{}

  We note that $F_*^{e}$ gives us an exact covariant functor from the category of $R$-modules to itself. This is nothing but the usual restriction of scalars functor associated to the ring homomorphism $F^{e}: R \rightarrow R$. 

{}

For an ideal $I\subset R$, the notation $I^{[p^e]}$ denotes the ideal generated by the $p^e$-th powers of the elements of $I$. Equivalently, $I^{[p^e]}$ is the expansion of $I$ under the Frobenius map, that is, $I^{[p^e]} = IF^e_*R$ as subsets of $R$.

{}

  The image of $F^e$ is the subring  $R^{p^e} \subset R$ of $p^e$-th powers. If $R$ is reduced (which is equivalent to the injectivity of Frobenius),
statements about the $R$-module $F_*^eR$ are equivalent to statements about the  $R^{p^e}$-module $R$.

\begin{definition} 
A ring $R$ of characteristic $p$ is \textbf{F-finite} if $F:R\rightarrow F_*R$ is a finite map of rings, or equivalently, if $R$ is a finitely generated $R^p$-module.  Note that $F:R\rightarrow F_*R$ is a finite map if and only if $F^e:R \rightarrow F_*^{e}R$ is a finite map for all $e>0$. 
\end{definition}

{}

F-finite rings are ubiquitous. For example, every perfect field is F-finite, and a finitely generated algebra over an F-finite ring is F-finite. Furthermore, F-finiteness is preserved under homomorphic images, localization and completion. This means that nearly every ring classically arising in algebraic geometry is F-finite. However, valuation rings even of F-finite fields are often not F-finite.

\subsection{F-purity and Frobenius splitting}\label{F-purity and Frobenius splitting}
We first review purity and splitting for maps of modules over an arbitrary commutative ring $A$, not necessarily Noetherian or of prime characteristic. 
A map of $A$-modules $M\overset{\varphi}\to N$ is \textbf{pure} if for any $A$-module $Q$, the induced map
$$M \otimes_A Q \rightarrow N \otimes_A Q$$
is injective.  The map $M\overset{\varphi}\to N$ is \textbf{split} if $\varphi$ has a left inverse in the category of $A$-modules. Clearly, a split map is always pure. 
Although it is not obvious, the converse holds under a weak hypothesis:

\begin{lemma}\label{splitvspure} \cite[Corollary 5.2]{HR1}
Let $M\overset{\varphi}\to N$ be a pure map of $A$-modules where $A$ is a  commutative ring. Then $ \varphi$
is split if the cokernel $N/\varphi(M)$ is finitely presented.
\end{lemma}

\begin{definition} 
\label{basic char p notions} Let $R$ be an arbitrary commutative ring of prime characteristic $p$.
\begin{enumerate}
\item[(a)] The ring $R$ is \textbf{Frobenius split} if the map $F:R\rightarrow F_*R$ splits as a map of $R$-modules, that is, there exists an $R$-module map $F_*R \rightarrow R$ such that the composition
$$R \xrightarrow{F} F_*R \rightarrow R$$
is the identity map.
\item[(b)] The ring $R$ is \textbf{F-pure} if $F:R \rightarrow F_*R$ is a pure map of $R$-modules.
\end{enumerate}
\end{definition}

{}
 
A Frobenius split ring is always F-pure. The converse is also true under modest hypothesis:

{}

\begin{corollary}\label{equivalence}
\label{equivalence of splitting and purity}
A Noetherian F-finite ring of characteristic $p$ is Frobenius split if and only if it is $F$-pure.
\end{corollary}

\begin{proof}
The F-finiteness hypothesis implies that $F_*R$ is a finitely generated $R$-module. So a quotient of $F_*R$ is also finitely generated. Since a finitely generated module over a Noetherian ring is finitely presented, the result follows from Lemma \ref{splitvspure}.
\end{proof}

\subsection{F-finiteness and Excellence}\label{F-finiteExcellence}
Although we are mainly concerned with non-Noetherian rings in this paper, it is worth pointing out the following curiosity for readers familiar with Grothendieck's concept of an {\it excellent ring,} a particular kind of Noetherian ring expected to be the most general setting for many algebro-geometric statements \cite[D\'efinition 7.8.2]{Groth}.

\begin{proposition}\label{F-finiteExcellent} A Noetherian domain is  F-finite if and only if it is excellent and its fraction field is F-finite.
\end{proposition}

\begin{proof} If $R$ is F-finite with fraction field $K$,  then also $R \otimes_{R^p} K^p \cong K$  is finite over $K^p$, so the fraction field of $R$ is F-finite. Furthermore, 
Kunz showed that F-finite Noetherian rings are excellent \cite[Theorem 2.5]{Kunz1}.

 We need to show that an excellent Noetherian domain with F-finite fraction field is F-finite.
 We make use of the following well-known property{\footnote{sometimes called the {\textbf{Japanese}} or {\textbf{N2}} property.} of an excellent domain $A$}:  the integral closure of $A$ in any finite extension of its fraction field is {\it finite} as a $A$-module \cite[IV, 7.8.3 (vi)]{Groth}.  The ring $R^p$ is excellent because it is isomorphic to $R$, and its fraction field is $K^p$. Since $K^p\hookrightarrow K$ is finite,   the  integral closure $S$  of $R^p$ in   $K$ 
is a finite $R^p$-module.   But clearly  $R \subset S$, so $R$ is also a finitely generated $R^p$ module, since submodules of a Noetherian module over a Noetherian ring are Noetherian. That is, $R$ is F-finite.
\end{proof}

Using this observation, we can clarify the relationship between F-purity and Frobenius splitting in an important class of rings.

\begin{corollary}\label{equivInExcellent} For an excellent  Noetherian domain whose fraction field is F-finite,  Frobenius splitting is equivalent to F-purity.
\end{corollary}

\begin{proof}[Proof of Corollary]
Our hypothesis implies F-finiteness, so splitting and purity are equivalent by Lemma \ref{splitvspure}.
\end{proof}

\section{Flatness and Purity  of Frobenius in Valuation Rings.}
\label{flatness and purity for valuations}

  Kunz showed that for a Noetherian ring of characteristic $p$, the Frobenius map is flat if and only if the ring is regular \cite[Theorem 2.1]{Kunz}.
In this section, we show how standard results on valuations yield the following  result:
 
\begin{Theorem}
\label{flatness and faith flatness of Frobenius for Valuation rings}
Let $V$ be a valuation ring of characteristic $p$. Then the Frobenius map $F: V \rightarrow F_*V$ is faithfully flat. \end{Theorem}

{} 

 This suggests that we can imagine a valuation ring to be ``regular'' in some sense. 
Of course, a Noetherian valuation ring is either a field or  a one dimensional regular local ring--- but because valuation rings are rarely Noetherian, Theorem  \ref{flatness and faith flatness of Frobenius for Valuation rings} is not a consequence of Kunz's theorem.

{}

Theorem \ref{flatness and faith flatness of Frobenius for Valuation rings} follows  from the following general result, whose proof we include for the sake of completeness.

\begin{Lemma}\cite[VI.3.6, Lemma 1]{Bourbaki}.
\label{finitely generated torsion free modules over valuation rings}
A finitely-generated, torsion-free module over a  valuation ring  is free. In particular, a torsion free module over a valuation ring is flat. 
\end{Lemma}

\begin{proof} Let $M\neq 0$ be a finitely generated, torsion-free $V$-module. Choose a minimal set of generators $\{m_1,\dots,m_n\}$.  If there is a non-trivial relation among these generators, 
then there exists $v_1, \dots, v_n \in V$ (not all zero) such that $v_1m_1 + \dots + v_nm_n = 0$. Re-ordering if necessary, we may assume that $v_1$ is minimal among (non-zero) coefficients, that is, $(v_i) \subset (v_1)$ for all $i \in \{1,\dots,n\}$.  Then for each $i > 1$, there exists $a_i \in V$ such that $v_i = a_iv_1$. This implies that 
 $$v_1(m_1 + a_2m_2 + \dots + a_nm_n) = 0.$$ 

  Since $v_1 \neq 0$ and $M$ is torsion free, we get 
$$m_1 + a_2m_2 + \dots + a_nm_n = 0.$$

   Then $m_1 = -(a_2m_2 + \dots + a_nm_n)$. So $M$ can be generated by the smaller set $\{m_2,\dots,m_n\}$ which contradicts the minimality of $n$. Hence $\{m_1,\dots,m_n\}$ must be a free generating set.

{} 

  The second statement follows by considering a torsion-free module as a directed union of its finitely generated submodules, since a directed union of flat modules is flat \cite[I.2.7 Prop 9]{Bourbaki} 
\end{proof}

\begin{proof}[Proof of Theorem \ref{flatness and faith flatness of Frobenius for Valuation rings}]
Observe that 
$F_*V$ is a torsion free $V$-module. So by Corollary \ref{finitely generated torsion free modules over valuation rings}, the module $F_*V$ is  flat, which means the Frobenius map is flat. 
To see that Frobenius is faithfully flat, we need only check that  $mF_*V \neq F_*V$ for $m$ the maximal ideal of $V$   
\cite[I.3.5 Prop 9(e)]{Bourbaki}.
But this is clear: the element  $1 \in F_*V$ is not in 
$mF_*V $, since $1\in V$ is not in the ideal $ m^{[p]}.$
\end{proof}

\begin{Corollary}
\label{Frobenius is pure}
Every valuation ring of characteristic $p $ is $F$-pure.
\end{Corollary}

\begin{proof} Fix a valuation ring $V$ of characteristic $p$. 
We have already seen that the Frobenius map $V \rightarrow F_*V$ is faithfully flat (Theorem \ref{flatness and faith flatness of Frobenius for Valuation rings}). 
But any faithfully flat map of rings $A\rightarrow B$ is pure as a map of $A$-modules \cite[I.3.5 Prop 9(c)]{Bourbaki}. 
\end{proof}

\section{F-finite Valuation Rings}

  In this section, we investigate F-finiteness in valuation rings. We first prove Theorem \ref{F-finite iff free} characterizing F-finite valuation rings as those $V$ for which $F_*V$ is a {\it free} $V$-module. We then prove a numerical  characterization of F-finiteness in terms of ramification index and residue  degree for extensions of valuations under Frobenius in Theorem \ref{F-finiteness in terms of ramification and residue degree}. This characterization is useful for constructing interesting examples, and later for showing that F-finite valuations are Abhyankar.

\subsection{Finiteness and Freeness of Frobenius}   For any domain $R$ of characteristic $p$, we have already observed (see the proof of Proposition \ref{F-finiteExcellent}) that a necessary condition for F-finiteness is the F-finiteness of its fraction field. For this reason, we investigate F-finiteness of valuation rings only in   F-finite ambient fields.

\begin{theorem}
\label{F-finite iff free}
Let $K$ be an F-finite field. A valuation ring $V$ of $K$ is F-finite if and only if $F_*V$ is a free $V$-module. 
\end{theorem}

\begin{proof}
First assume $F_*V$ is free over $V$. Since $K \otimes_R F_*V \cong F_*K$ as $K$-vector spaces, the rank of $F_*V$ over $V$ must be the same as the rank of $F_*K$ over $K$, namely the degree $[F_*K:K] = [K:K^p]$. Since $K$ is F-finite, this degree is  finite,
 and so $F_*V$ is a free $V$-module of finite rank. In particular, $V$ is F-finite. 
 
 \vspace{1mm}

  Conversely, suppose that $V$ is F-finite. Then $F_*V$ is a finitely generated, torsion-free $V$-module. So it is free by Lemma \ref{finitely generated torsion free modules over valuation rings}.
\end{proof}

\begin{corollary}
\label{F-finiteness implies F-split} 
An F-finite  valuation ring  is  Frobenius split.
\end{corollary}

\begin{proof} One of  the rank one free summands of $F_*V$ is the copy of $V$ under $F$, so this copy of $V$ splits off $F_*V$. Alternatively, since $V\rightarrow F_*V$ is pure, we can use Lemma \ref{splitvspure}: the cokernel of $V\rightarrow F_*V$ is finitely presented because it is finitely generated (being a quotient of the finitely generated $V$-module $F_*V$) 
and the module of relations is finitely generated (by  $1 \in F_*V$). 
\end{proof}

\begin{remark}\label{Linquan2}
The same argument shows that any module finite extension $V\hookrightarrow S$ splits---in other words, every valuation ring is a {\it splinter} in the sense of \cite{Ma}; see also \cite[Lemma 1.2]{Ma}. 
\end{remark}

{}

\subsection{Frobenius Splitting in the Noetherian case.}\label{FsplitExcellent}
We can say more for  Noetherian valuation rings. First we make a general observation about F-finiteness in Noetherian rings:

\begin{theorem}\label{Noeth} For a  Noetherian domain whose fraction field is F-finite,  Frobenius splitting implies F-finiteness (and hence excellence).
\end{theorem}

{}
Before embarking on the proof, we point out a consequence for valuation rings:
\begin{corollary}\label{DVRFsplitChar} For a  discrete valuation ring $V$ whose fraction field is F-finite, the following are equivalent:
\begin{itemize}
\item[(i)] $V$ is Frobenius split;
\item[(ii)]  $V$ is F-finite;
\item[(iii)]  $V$ is excellent.
\end{itemize}
\end{corollary}

\begin{proof}[Proof of Corollary]
A DVR is Noetherian, so equivalence of (i) and (ii) follows from combining Theorem \ref{Noeth} and Corollary \ref{F-finiteness implies F-split}. The equivalence with excellence follows from Proposition  \ref{F-finiteExcellent}.
\end{proof}

\begin{remark} We have proved that all valuation rings are F-pure. However, not all valuation rings, even discrete ones on $\mathbb F_p(x,y)$, are Frobenius split, as 
 Example \ref{A discrete valuation ring which is not F-finite} below shows. \end{remark}
 
The  proof of Theorem \ref{Noeth} relies on the following lemma:

\begin{lemma} \label{F-finiteFSplit}
A Noetherian domain with F-finite fraction field is F-finite if and only there exists $\phi \in Hom_{R^p}(R, R^p)$ such that $\phi(1) \neq 0$.
\end{lemma}

\begin{proof}[Proof of Lemma] 
Assuming such $\phi$ exists, we first observe that the canonical map to the double dual
$$
R \rightarrow R^{\vee\vee}:= \operatorname{Hom}_{R^p}(\operatorname{Hom}_{R^p}(R, R^p), R^p) 
$$
is injective.
Indeed, let $x \in R$ be a non-zero element. It suffices to show that there exists $f \in R^{\vee}:= \operatorname{Hom}_{R^p}(R, R^p)$ such that $f(x) \neq 0$. Let $f = \phi \circ x^{p-1}$, where $x^{p-1}$ is the $R^p$-linear map $R\rightarrow R$ given by multiplication by $x^{p-1}$.
 Then $f(x) = \phi(x^p) = x^p \phi(1) \neq 0$. This shows that the double dual map is injective.

 Now, to show that $R$ is a finitely generated $R^p$-module, it suffices to show that the larger module  $R^{\vee\vee} $ is finitely generated. For this it suffices to show that 
  $R^{\vee}$ is  a finitely generated $R^p$-module, since the dual of a finitely generated module is finitely generated.

We now show that $R^{\vee}$ is  finitely generated.  Let $M$ be a maximal free $R^p$-submodule of  $R$. Note that $M$ has finite rank (equal to $[K:K^p]$, where $K$ is the fraction field of $R$) and that $R/M$ is a torsion $R^p$-module. 
Since the dual of a torsion module is zero, dualizing the exact sequence $0 \rightarrow M \rightarrow R \rightarrow R/M \rightarrow 0$ induces an injection
$$
R^{\vee} := \operatorname{Hom}_{R^p}(R, R^p)\hookrightarrow \operatorname{Hom}_{R^p}(M, R) =  M^{\vee}.
$$ Since $M$ is a  finitely generated $R^p$-module, also $M^{\vee}$, and hence its submodule $R^{\vee}$ is finitely generated ($R$ is Noetherian). This completes the proof that $R$ is F-finite.

For the converse, fix any $K^p$-linear splitting $\psi: K \rightarrow  K^p$. Restricting to $R$ produces an $R^p$-linear map to $K^p$. Since $R$ is finitely generated over $R^p$, we can multiply by some non-zero element $c$ of $R^p$ to produce a non-zero map $\phi:R\rightarrow R^p$ such that $\phi(1) = c \neq 0$.  The proof of Lemma \ref{F-finiteFSplit} is complete.
\end{proof}

\begin{proof}[Proof of Theorem \ref{Noeth}]  Let $R$ be a domain  with F-finite fraction field. A
 Frobenius splitting is a map $\phi\in \operatorname{Hom}_{R^p}(R, R^p)$ such that $\phi(1) = 1$.  Theorem \ref{Noeth}  then follows immediately from Lemma \ref{F-finiteFSplit}.
 \end{proof}




\subsection{A numerical Criterion for F-finiteness} 
 Consider the extension 
$$K^p \subseteq K$$
where $K$ any field of characteristic $p$.
For any valuation $v$ on $K$, let $v^p$ denote the restriction to  $K^p$.   
 We next characterize F-finite valuations in terms of the ramification index and residue degree of $v$ over $v^p$:
   
{}
   
\begin{theorem}
\label{F-finiteness in terms of ramification and residue degree}
A valuation ring $V$ of an F-finite field $K$  of prime characteristic $p$  is F-finite if and only if $$e(v/v^p)f(v/v^p) = [K:K^p],$$ where $v$ is the corresponding valuation on $K$ and $v^p$ is its restriction to $K^p$. 
\end{theorem}

\begin{proof}
{}

First note that  $v$ is the only valuation of $K$ extending $v^p$. Indeed, $v$ is uniquely determined by its values on elements of $K^p$, since $v(x^p) = p v(x)$ and the value group of $v$ is torsion-free. Furthermore, the valuation ring of $v^p$ is easily checked to be $V^p$.

Observe that  $V$ is the integral closure of $V^p$ in $K$.  Indeed, since $V$ is a valuation ring, it is integrally closed in $K$, but it is also obviously integral over $V^p$.  
We now apply Lemma \ref{key theorem for F-finiteness results}.  Since there is only one valuation extending $v^p$,  the  inequality 
$$e(v/v^p)f(v/v^p) \leq [K:K^p]$$
will be an equality if and only if the integral closure of $V^p$ in $K$, namely $V$,  is finite over $V^p$. 
\end{proof}

{}

  The following  simple consequence has useful applications to the construction of interesting examples of F-finite  and non F-finite valuations:

{}

\begin{corollary}
\label{corollary to F-finiteness in terms of ramification and residue degree}
Let  $V$ be a valuation ring of an $F$-finite field $K$ of characteristic $p$. If $e(v/v^p)$ = $[K:K^p]$ or $f(v/v^p)$ = $[K:K^p]$, then $V$ is $F$-finite.
\end{corollary}
\vspace{1mm}

\begin{remark}
  Theorem \ref{F-finiteness in terms of ramification and residue degree} and its corollary are easy to apply, because the
   ramification index and residue degree for the extension $V^p \hookrightarrow V$ can be computed  in practice. Indeed,
      since $\Gamma_{v^p}$ is clearly the subgroup  $p\Gamma_v$ of $\Gamma_v$, we see that  
   \begin{equation}
   \label{ramindex}
   e(v/v^p) = [\Gamma_v:p\Gamma_v].
   \end{equation}  
   Also, the local map $V^p \hookrightarrow V$ induces the residue field extension $\kappa(v^p) \hookrightarrow \kappa(v)$, which identifies the field $\kappa(v^p) $ with the subfield 
   $(\kappa(v))^p.$ This means that 
   \begin{equation}
   \label{resdegree}
   f(v/v^p) = [\kappa(v):\kappa(v)^p].
   \end{equation}

{}

\end{remark}

{}

\subsection{Examples of Frobenius Split Valuations} We can use our characterization of F-finite valuations to easily give examples of valuations on $\mathbb F_p(x, y)$ that are non-discrete but Frobenius split.
\begin{example}
Consider the  rational function field $K=k(x, y)$ over a perfect field  $k$ of characteristic $p$.  For an   irrational number $\alpha \in \mathbb{R}$, let $\Gamma$ be the ordered additive subgroup of $\mathbb R$ generated by $1$ and $\alpha$. Consider the unique valuation  $v:K^{\times}\rightarrow \Gamma$ 
 determined by
$$v(x^iy^j ) =  {i + j\alpha },$$ and let  $V$ be the corresponding valuation ring.  
Since $\Gamma \cong \mathbb{Z} \oplus \mathbb{Z}$ via the map which sends $a + b\alpha \mapsto (a, b)$, we see that the value group of $v^p$ is $p\Gamma \cong p(\mathbb{Z} \oplus \mathbb{Z})$. Hence 
$$e(v/v^p) = [\Gamma: p\Gamma] = p^2 = [K:K^p].$$ 
So $V$ is F-finite by Corollary  \ref{corollary to F-finiteness in terms of ramification and residue degree}. Thus $V$ is also Frobenius split by Corollary \ref{F-finiteness implies F-split}.
\end{example}

\begin{example}\label{lex}
  Consider the lex valuation on the rational function field  $K = k(x_1, x_2, \dots, x_n)$ over a perfect field  $k$ of characteristic $p$. This is the valuation $v:K^{\times}\rightarrow \mathbb Z^n$  on $K/k$ defined by 
sending a monomial $x_1^{a_1}\dots x_n^{a_n}$ to  $(a_1, \dots, a_n) \in \mathbb Z^{\oplus n}$, where $\Gamma =  \mathbb{Z}^{\oplus n}$  is ordered lexicographically.  Let $V$ be the corresponding valuation ring. The value group of $V^p$ is $p\Gamma$,  so $e(v/v^p) = [\Gamma:p\Gamma] = p^n = [K:K^p]$. As in the previous example,  Corollary  \ref{corollary to F-finiteness in terms of ramification and residue degree} implies that $V$ is $F$-finite, and so again F-split. 
\end{example}

{}
\subsection{Example of a non-Frobenius Split Valuation}
 Our next example shows that discrete valuation rings are not always F-finite, even in the rational function field $\mathbb F_p(x,y)$. This is adapted from \cite[Example on pg 62]{Z-S}, where it is credited to F. K. Schmidt.

\begin{example}
\label{A discrete valuation ring which is not F-finite}
Let $\mathbb{F}_p((t))$ be the fraction field of the discrete valuation ring  $\mathbb{F}_p[[t]]$  of power series in one variable.  
Since the field of rational functions $\mathbb{F}_p(t)$ is countable, the uncountable field $\mathbb{F}_p((t))$ can not be algebraic over $\mathbb{F}_p(t)$.
 So we can find  some power series $$f(t) = \sum_{n=1}^{\infty} a_nt^n$$ in $\mathbb{F}_p[[t]]$  transcendental over $\mathbb{F}_p(t)$. 
{}

  Since $t$ and $f(t) $ are algebraically independent, there is an {\it injective} ring map $$\mathbb{F}_p[x,y] \hookrightarrow \mathbb{F}_p[[t]]\,\,\,\, {\text{ sending}}\,\,\,x \mapsto t \,\,\,\, {\text{and}}\,\,\,\,\, y \mapsto f(t)$$  which induces an extension of fields $$\mathbb{F}_p(x,y) \hookrightarrow \mathbb{F}_p((t)).$$ 
Restricting the $t$-adic valuation  on $\mathbb{F}_p((t))$ to  the subfield $\mathbb{F}_p(x,y)$ produces a discrete valuation  $v$ of $\mathbb F_p(x,y)$. Let $V$ denote its valuation ring.
{}

   We claim that   $V$ is not F-finite, a statement we can verify with Theorem \ref{F-finiteness in terms of ramification and residue degree}.
Note that $L= \mathbb{F}_p(x,y)$ is F-finite, with  $[L:L^p] =  p^2$.  Since the value group $\Gamma_v$ is $\mathbb{Z}$, we see that $$e(v/v^p) = [\Gamma_v:p\Gamma_v] = p.$$ 
On the other hand, to compute the residue degree $f(v/v^p)$, we must understand the field extension $\kappa(v)^p \hookrightarrow \kappa(v)$.  Observe that
for an element $u\in \mathbb F_p(x,y)$ to be in $V$, its image in $\mathbb F_p((t))$ must be a power series of the form $$\sum_{n=0}^{\infty}b_nt^n$$ where $b_n \in \mathbb{F}_p$. Clearly 
$$v(u-b_0) > 0$$ which means that the class of  $u = \sum_{n=0}^{\infty}b_nt^n$ in $\kappa(v)$ is equal to the class of $b_0$ in $\kappa(v)$. This implies that $\kappa(v) \cong \mathbb{F}_p,$ so that $ [\kappa(v):\kappa(v)^p] = 1$. That is,  $f(v/v^p)  =1$.

{}

  Finally, we then have that 
 $$e(v/v^p)f(v/v^p) = p \neq p^2 = [\mathbb{F}_p(x,y) :(\mathbb{F}_p(x,y) )^p].$$ So $V$ \textbf{cannot} be F-finite by  Theorem \ref{F-finiteness in terms of ramification and residue degree}. Thus this Noetherian ring is neither Frobenius split nor excellent by Corollary \ref{DVRFsplitChar}.

\end{example}

\vspace{1mm}

\subsection{Finite Extensions}  Frobenius properties of valuations are largely preserved under finite extension. First note that 
 if $K\hookrightarrow L$ is a finite extension of F-finite fields, then $[L:L^p] = [K:K^p]$;  this follows immediately from the 
 commutative diagram of fields
$$\xymatrix{
L  &K\ar@{_{(}->}[l]\\
L^p \ar@{_{(}->}[u] & K^p.\ar@{_{(}->}[l] \ar@{_{(}->}[u]
}$$
To wit,  $  [L:K^p] = [L:K][K:K^p] = [L:L^p][L^p:K^p]$ and $ [L:K] = [L^p:K^p]$, so  that $[L:L^p] = [K:K^p]$.
Moreover,

\begin{proposition}
\label{F-finiteness preserved under extensions}
Let $K\hookrightarrow L$ be a finite extension of F-finite fields of characteristic $p$.  Let $v$ be a valuation on $K$ and $w$ an extension of $v$ to $L$. Then:
\begin{enumerate}
\item[(i)] The ramification indices $e(v/v^p)$ and $e(w/w^p)$ are equal. 
\item[(ii)] The residue degrees $f(v/v^p)$ and $f(w/w^p)$ are equal.
\item[(iii)] The valution ring for $v$ is $F$-finite if and only if the valuation ring for $w$ is $F$-finite.
\end{enumerate}
\end{proposition}

\begin{proof}
By (\ref{ramification/residue for finite extensions is finite}), we have $$[\Gamma_w:\Gamma_v][\kappa(w):\kappa(v)] \leq [L:K], $$ so both $[\Gamma_w: \Gamma_v]$ and $[\kappa(w):\kappa(v)]$ are finite. Of course, we also know that the ramification indices  $e(w/w^p) = [\Gamma_w: p\Gamma_w]$ and $e(v/v^p)  = [\Gamma_v: p\Gamma_v]$ are finite, as are
 the residue degrees $f(w/w^p) = [\kappa(w):\kappa(w)^p]$ and $f(v/v^p) = [\kappa(v):\kappa(v)^p].$

{}

  (i) In light of (\ref{ramindex}),
 we need to show that  $[\Gamma_w: p\Gamma_w] = [\Gamma_v: p\Gamma_v]$.
Since $\Gamma_w$ is torsion-free, multiplication by $p$ induces an isomorphism $\Gamma_w \cong p\Gamma_w$, under which the subgroup  $\Gamma_v$  corresponds to $p\Gamma_v$.  
Thus $[p\Gamma_w:p\Gamma_v] = [\Gamma_w: \Gamma_v]$. Using the commutative diagram of  finite index abelian subgroups
$$\xymatrix{
\Gamma_w  &\Gamma_v\ar@{_{(}->}[l]\\
p\Gamma_w \ar@{_{(}->}[u] & p\Gamma_v, \ar@{_{(}->}[l] \ar@{_{(}->}[u]
}$$ we see that  $[\Gamma_w:p\Gamma_w][ p\Gamma_w : p\Gamma_v] = [\Gamma_w:\Gamma_v] [\Gamma_v:p\Gamma_v].$ Whence $[\Gamma_w:p\Gamma_w] = [\Gamma_v:p\Gamma_v]$.

{}

  (ii) In light of (\ref{resdegree}),
we need to show that  $[\kappa(w):\kappa(w)^p] = [\kappa(v):\kappa(v)^p].$ 
We have $[\kappa(w)^p:\kappa(v)^p]  = [\kappa(w): \kappa(v)]$, so the result follows from computing the extension degrees in the 
commutative diagram of finite field extensions
$$\xymatrix{
\kappa(w)  &\kappa(v)\ar@{_{(}->}[l]\\
\kappa(w)^p \ar@{_{(}->}[u] & \kappa(v)^p. \ar@{_{(}->}[l] \ar@{_{(}->}[u]
}$$

{}

  (iii) By (i) and (ii) we get $e(w/w^p) = e(v/v^p)$ and $f(w/w^p) = f(v/v^p)$. Therefore 
$$e(w/w^p)f(w/w^p) = e(v/v^p)f(v/v^p).$$

  Since also $[L:L^p] = [K:K^p]$, we see using Theorem \ref{F-finiteness in terms of ramification and residue degree} that $w$ is $F$-finite if and only if $v$ is $F$-finite.
\end{proof}

\section{F-finiteness in Function Fields}\label{Abhyankar}

  An important class of fields are {\it function fields over a ground field $k$}. By definition, a field $K$ is a function field over $k$ if it is a finitely generated field extension of $k$. These are the fields that arise as function fields of varieties over a (typically algebraically closed) ground field $k$. What more can be said about valuation rings in this important class of fields? 

{}
  
We saw in Example \ref{A discrete valuation ring which is not F-finite}  that not every valuation of an F-finite function field is F-finite. However, the following theorem gives a nice  characterization of those that are.

\begin{Theorem}
\label{F-finiteness equivalent to Abhyankar}
Let $K$ be a finitely generated field extension of an F-finite ground field $k$.
The following are equivalent for a valuation $v$ on $K/k$:
\begin{enumerate}
\item[(i)] The valuation $v$ is Abhyankar.
\item[(ii)] The valuation ring $R_v$ is F-finite.
\item[(iii)] The valuation ring $R_v$ is a free $R_v^p$-module.
\end{enumerate}
Furthermore,  when these equivalent conditions hold, it is also true that $R_v$ is Frobenius split. 
\end{Theorem}

Since Abhyankar valuations have finitely generated value groups and residue fields, the following corollary holds.

\begin{Corollary}
\label{F-finiteness implies finitely generated value group}
An F-finite valuation of a function field over an F-finite field  $k$ has 
 finitely generated value group and its residue field is a finitely generated field extension of $k$. 
\end{Corollary}

  For example, valuations whose value groups are $\mathbb Q$ can never be F-finite.

\begin{Remark}\label{excellence}
In light of Proposition \ref{F-finiteExcellent}, we could add a fourth  item to the list of equivalent conditions in Theorem \ref{F-finiteness equivalent to Abhyankar} in the {\it Noetherian} case: the valuation $R_v$ is excellent. The theorem says that the only discrete valuation rings (of function fields) that are F-finite are the divisorial valuation rings  or equivalently, the  excellent DVRs.
\end{Remark}

{}
   To prove Theorem \ref{F-finiteness equivalent to Abhyankar}, first recall that the equivalence of (ii) and (iii) was already established in Theorem \ref{F-finite iff free}. The point is to connect these conditions with the Abyhankar property. 
Our strategy is to use Theorem \ref{F-finiteness in terms of ramification and residue degree}, which tells us that a valuation $v$ on $K$ is F-finite if and only if 
$$e(v/v^p) f(v/v^p) = [K:K^p].$$  
 We do this by 
proving two  propositions---one comparing the rational rank of $v$ to the ramification index $e(v/v^p)$, and the other comparing the transcendence degree of $v$ to the residue degree $f(v/v^p)$.

{}

\begin{Proposition}
\label{ebound}
Let $v$ be a valuation of rational rank $s$ on an F-finite field $K$. 
Then $$ e(v/v^p) \leq p^s,$$ with equality when the value group $\Gamma_v$ is finitely generated.
\end{Proposition}

\begin{proof} 
To see that equality holds when $\Gamma_v$ is finitely generated, note that in this case, $\Gamma_v \cong \mathbb Z^{\oplus s}$. So $\Gamma_v/p\Gamma_v \cong (\mathbb Z/p\mathbb Z)^{\oplus s}$, which has cardinality $p^s$. 
That is, $e(v/v^p)  = p^s.$

{}

  It remains to consider the case where $\Gamma$ may not be finitely generated. Nonetheless, since $e(v/v^p)$ is finite (C.f. \ref{ramification/residue for finite extensions is finite}), we do know that 
$[\Gamma_v: p\Gamma_v] = e(v/v^p)$ is finite. So the proof of Proposition \ref{ebound} comes down to the following simple lemma about abelian groups:

\begin{Lemma} 
Let $\Gamma$ be a torsion free abelian group of rational rank $s$.  Then $[\Gamma:p\Gamma] \leq p^s$. 
\end{Lemma}

{}

It suffices to show that $\Gamma/p\Gamma$ is a vector space over $\mathbb{Z}/p\mathbb{Z}$ of dimension $\leq s$. For if $d$ is the dimension of $\Gamma/p\Gamma$, then $[\Gamma:p\Gamma] = p^d$.  So
let  $t_1, \dots, t_n$ be elements of $\Gamma$ whose classes modulo $p\Gamma$ are linearly independent over $\mathbb{Z}/p\mathbb{Z}$. 
 Then we claim that the $t_i$ are $\mathbb Z$-independent elements of $\Gamma$. 
  Assume to the contrary that there is some non-trivial relation 
 $a_1t_1 + \dots + a_nt_n = 0$,  for some integers $a_i$. Since $\Gamma$ is torsion-free, we can assume without loss of generality, that at least  one $a_j$ is not divisible by $p$. But now modulo $p\Gamma$, this relation produces a non-trivial relation on classes of the $t_i$ in $\Gamma/p\Gamma$, contrary to the fact that these are linearly independent. This shows that any $\mathbb{Z}/p\mathbb{Z}$-linearly independent subset of $\Gamma/p\Gamma$ must have cardinality at most $s$. Then the lemma, and hence  Proposition \ref{ebound}, are proved. 
\end{proof}

\begin{Proposition}
\label{fbound} 
Let $K$ be a  finitely generated field extension of an F-finite ground field $k$. Let $v$ be a valuation of transcendence degree $t$ on  $K$ over $k$. 
Then 
$$f(v/v^p) \leq p^t[k:k^p],$$ 
with equality when $\kappa(v)$ is finitely generated over $k$.
\end{Proposition} 

\begin{proof} 
The second statement follows immediately from the following well-known fact, whose proof is an easy computation:

\begin{Lemma}
\label{simple field theory lemma} 
A finitely generated field $L$ of characteristic $p$ and transcendence degree $n$ over $k$ satisfies $[L:L^p] = [k:k^p]p^n.$
\end{Lemma}

{}

  It remains to consider the case where $\kappa(v)$ may not be finitely generated. Because $K/k$ is a function field, Abhyankar's inequality (\ref{Abineq}) guarantees that the transcendence degree of $\kappa(v)$ over $k$ is finite. Let $x_1, \dots, x_t$ be a transcendence basis.
There is a factorization 
$$ k \hookrightarrow k(x_1, \dots, x_t) \hookrightarrow \kappa(v)$$
where the second inclusion is algebraic. 
The proposition follows immediately from:

\begin{Lemma}
\label{generalization of a previous result}
If $L' \subseteq L$ is an algebraic extension of F-finite fields, then $[L:L^p] \leq [L':L'^p]$.
\end{Lemma}

  To prove this lemma, recall that Proposition \ref{F-finiteness preserved under extensions} ensures that $[L:L^p] = [L':L'^p]$  when $L' \subseteq L$ is  {\it finite.} So suppose $L$ is algebraic but not necessarily finite over $L'$.  Fix a basis $\{\alpha_1, \dots, \alpha_n\}$ for $L$ over $L^p$, and consider the intermediate field 
$$L' \hookrightarrow  L' (\alpha_1,\dots,\alpha_n) \hookrightarrow L.$$
Since each $\alpha_i$ is algebraic over $L'$, it follows that $\tilde L := L'(\alpha_1,\dots,\alpha_n)$ is finite over  $L'$, so again $[\tilde L:\tilde L^p] = [L':L'^p] $ by Proposition \ref{F-finiteness preserved under extensions}. Now observe that $\tilde L^p \subset L^p$, and so the $L^p$-linearly independent set  $\{\alpha_1, \dots, \alpha_n\}$ is also linearly independent over $\tilde L^p$. This means that $[L:L^p] \leq [\tilde L:\tilde L^p] $ and hence $[L: L^p] \leq [L':L'^p].$
This proves Lemma \ref{generalization of a previous result}.

{}

  Finally, Proposition \ref{fbound} is proved by applying Lemma \ref{generalization of a previous result} to the inclusion 
$$L'=k(x_1, \dots, x_t) \hookrightarrow L=\kappa(v).$$
[Note that  $\kappa(v)$ is F-finite,  because
$[\kappa(v):(\kappa(v))^p] = f(v/v^p) \leq [K:K^p]$ from the  general inequality (\ref{ramification/residue for finite extensions is finite})).] So  we get $$f(v/v^p) = [\kappa(v):(\kappa(v))^p]  \leq [k(x_1, \dots, x_t): (k(x_1, \dots, x_t))^p] = p^t[k:k^p].$$
\end{proof}

\begin{proof}[Proof of Theorem \ref{F-finiteness equivalent to Abhyankar}]

It only remains to prove the equivalence of (i) and (ii). First assume $v$ is Abhyankar. Then its value group $\Gamma_v$ is finitely generated and its residue field $\kappa(v)$ is finitely generated over $k$. According to Proposition \ref{ebound}, we have $e(v/v^p) =  p^s$, where $s$ is the rational rank of $v$. According to Proposition \ref{fbound}, we have $f(v/v^p) = p^t[k:k^p]$, where $t$ is the transcendence degree of $v$. By definition of Abhyankar, $s+t = n$, where $n$ is the transcendence degree of $K/k$. But then
$$e(v/v^p)f(v/v^p) = (p^s)(p^t)[k:k^p] = p^{n}[k:k^p] = [K:K^p].$$
By Theorem \ref{F-finiteness in terms of ramification and residue degree}, we can conclude that $v$ is F-finite.

{}

  Conversely, we want to prove that a valuation $v$ with F-finite valuation ring $R_v$ is Abhyankar. 
Let $s$ denote the rational rank and $t$ denote the transcendence degree of $v$. 
From Theorem \ref{F-finiteness in terms of ramification and residue degree}, the F-finiteness of $v$ gives
$$e(v/v^p)f(v/v^p) = [K:K^p] = p^n[k:k^p].$$
Using the bounds  $p^s \geq e(v/v^p)$ and $p^t[k:k^p] \geq f(v/v^p)$ provided by Propositions \ref{ebound} and \ref{fbound}, respectively, we  substitute to get
$$(p^s)(p^t[k:k^p])   \geq  e(v/v^p) f(v/v^p) = p^n[k:k^p].$$
It follows that $s+t \geq n$. Then $s + t = n$ by (\ref{Abineq}), and $v$ is Abhyankar.
\end{proof}

\section{F-regularity}\label{F-regularity}

  An important class of F-pure rings are the  strongly F-regular rings. Originally, strongly F-regular rings were defined only in the Noetherian F-finite case. By definition, a 
Noetherian F-finite reduced ring $R$ of prime characteristic $p$  is strongly F-regular if for every non-zero-divisor $c$, there exists $e$ such that the 
map
$$R \rightarrow F_*^eR \,\,\,\,\,\,\,\,\,{\text{sending}}\,\,\,\, 1\mapsto c$$
splits in the category of $R$-modules  \cite{HH1}.
In this section, we show that by replacing the word "splits" with the words "is pure" in the above definition, we obtain a well-behaved notion of F-regularity in a broader setting.
Hochster and Huneke themselves suggested, but never pursued, this possibility in \cite[Remark 5.3]{HH4}.

 {}

  Strong F-regularity first arose as a technical tool in the theory of {\it tight closure}: Hochster and Huneke made use of it in their deep proof of the existence of  test elements \cite{HH4}. Indeed, the original motivation for (and the name of) strong F-regularity was born of a desire to better understand \textbf{weak F-regularity}, the property of a  Noetherian ring that all ideals are tightly closed. In many contexts, strong and weak F-regularity are known to be equivalent (see e.g. \cite{LySm} for the graded case, \cite{HH1} for the Gorenstein case)  but it is becoming clear that at least for many applications, strong F-regularity is the more useful and flexible notion.  Applications beyond  tight closure  include commutative algebra more generally \cite{AbLeus, Blickle, ST1, Schw2,  Smith-Zhang}, algebraic geometry \cite{ GLPSTZ,  hacon, Patak1,  SmSch, Smi2},  representation theory \cite{BrKu, MR, Ram, SmVan} and combinatorics \cite{BeMuRaSm}.

\vspace{1mm}

\subsection{Basic Properties of F-pure regularity}

We propose the following definition, intended to be a generalization of strong F-regularity to arbitrary commutative rings of characteristic  $p$,  not necessarily  F-finite or  Noetherian. 

\begin{definition}
\label{pure element and pure F-regularity}
Let $c$ be an element in a ring $R$  of prime characteristic $p$.  Then $R$ is said to be  \textbf{F-pure along $c$}
if there exists $e > 0$ such that the $R$-linear map
$$\lambda_c^e: R \rightarrow F_*^{e}R \,\,\,\,{\text{sending}} \,\,\, 1 \mapsto c$$
is a pure map of $R$-modules.
We say $R$ is \textbf{F-pure regular} if it is F-pure along every non-zerodivisor.
\end{definition}

  A ring $R$ is F-pure if and only if it is $F$-pure along the element $1$. Thus F-pure regularity is a substantial strengthening of  F-purity, requiring F-purity along {\it all} non-zerodivisors instead of just along the unit.

\begin{remark}
\label{some comments on the definition of pure F-regularity}
\begin{enumerate}
\item[(i)]  If $R$ is  Noetherian and F-finite, then the map $\lambda_c^e: R \rightarrow F_*^{e}R$ is pure if and only if it splits (by Lemma \ref{splitvspure}). So F-pure regularity for a Noetherian F-finite ring   is the same as strong F-regularity.
\item[(ii)] If $c$ is a zerodivisor, then the map $\lambda_c^e$ is never injective for any $e \geq 1$. In particular, a ring is never $F$-pure along a zerodivisor.
\item[(iii)] The terminology ``F-pure along $c$'' is chosen to honor Ramanathan's closely related notion of  ``Frobenius splitting along a divisor'' \cite{Ram}.
See \cite{Smi2}.
\end{enumerate}
\end{remark} 

{}
 
  The following proposition gathers up some basic properties of  F-pure regularity for arbitrary commutative rings.

\begin{proposition}
\label{basic properties of pure F-regularity}
Let $R$ be a commutative ring of characteristic $p$, not necessarily Noetherian or F-finite.
\begin{enumerate}
\item[(a)]  If $R$ is F-pure along some element, then  $R$ is F-pure. More generally,  if $R$ is  $F$-pure along a product $cd$, then  $R$ is $F$-pure along the factors  $c$ and $d$.
\item[(b)] If $R$ is  $F$-pure along some element, then $R$ is reduced.
\item[(c)] If $R$ is an F-pure regular ring with finitely many minimal primes, and $S \subset R$ is a multiplicative set, then $S^{-1}R$ is  F-pure regular. In particular,  F-pure regularity is preserved under localization in Noetherian rings, as well as in domains.
\item[(d)] Let $\varphi: R \rightarrow T$ be a pure ring map which maps non-zerodivisors of $R$ to non-zerodivisors of $T$. If $T$ is F-pure regular, then $R$ is F-pure regular. In  particular,   if $\varphi: R \rightarrow T$ is faithfully flat and $T$ is F-pure regular, then  $R$ is F-pure regular.
\item[(e)] Let $R_1, \dots, R_n$ be rings of characteristic $p$. If $R_1 \times \dots \times R_n$ is F-pure regular, then each $R_i$ is F-pure regular.
\end{enumerate}
\end{proposition}

{}

  The proof of Proposition \ref{basic properties of pure F-regularity} consists mostly of applying general facts about purity to the special case of the maps $\lambda^e_c$. For the convenience of the reader, we gather these basic facts together in one lemma:

\begin{lemma}
\label{properties of pure maps}
Let $A$ be an arbitrary commutative ring $A$, not necessarily Noetherian nor of characteristic $p$.
\begin{enumerate}
\item[(a)] If $M \rightarrow N$ and $N \rightarrow Q$ are pure maps of $A$-modules, then the composition $M\rightarrow N \rightarrow Q$ is also pure.

\item[(b)] If a composition $M \rightarrow N \rightarrow Q$ of $A$-modules  is pure, then $M \rightarrow N$ is pure.

\item[(c)] If $B$ is an $A$-algebra and $M \rightarrow N$ is pure map of $A$-modules, then $B \otimes_A M \rightarrow B \otimes_A N$ is a pure map of $B$-modules.

\item[(d)] Let $B$ be an $A$-algebra. If  $M \rightarrow N$ is a pure map of $B$-modules, then it is also pure as a map of $A$-modules.

\item[(e)] An $A$-module map $M \rightarrow N$ is pure if and only if for all prime ideals $\mathcal P \subset A$, $M_{\mathcal P} \rightarrow N_{\mathcal P}$ is pure.

\item[(f)] A faithfully flat map of rings  is pure.

 \item[(g)] If $(\Lambda, \leq)$ is a directed set with a least element $\lambda_0$, and $\{N_\lambda\}_{\lambda \in \Lambda}$ is a direct limit system of $A$-modules indexed by $\Lambda$ and $M \rightarrow N_{\lambda_0}$ is an $A$-linear map, then $M \rightarrow \varinjlim_{\lambda} N_{\lambda}$ is pure if and only if $M \rightarrow N_{\lambda}$ is pure for all $\lambda$.

\item[(h)] A map of modules  $A\rightarrow N$ over a Noetherian local ring $(A, m)$ is  pure if and only if  $E \otimes_A A \rightarrow E \otimes_A N$ is injective where $E$ is the injective hull of the residue field of $R$.

\end{enumerate}
\end{lemma}

\begin{proof}[Proof of Lemma \ref{properties of pure maps}]
Properties (a)-(d) follow easily from the definition of purity and elementary properties of tensor product. As an example, let us prove (d). If $P$ is an $A$-module, we want to show that 
$P \otimes_A M \rightarrow P \otimes_A N$ is injective. The map of $B$-modules
$$(P \otimes_A B) \otimes_B M \rightarrow (P \otimes_A B) \otimes_B N$$
is injective by purity of $M \rightarrow N$ as a map of $B$-modules. Using the natural $A$-module isomorphisms
$(P \otimes_A B) \otimes_B M \cong P \otimes_A M$ and 
$(P \otimes_A B) \otimes_B N \cong P\otimes_A N,$ 
we conclude that $
P \otimes_A M \rightarrow P \otimes_A N
$ is injective in the category of $A$-modules.
\vspace{1mm}

  Property (e) follows from (c) by tensoring with $A_p$ and the fact that injectivity of a map of modules is a local property. Property (f) follows from
 \cite[I.3.5, Proposition 9(c)]{Bourbaki}.  Properties (g)  and (h) are proved in \cite[Lemma 2.1]{HH3}.  \end{proof}

\begin{proof}[Proof of Proposition \ref{basic properties of pure F-regularity}]
(a) Multiplication by $d$ is an an $R$-linear map, so by restriction of scalars also
$$
F^e_*R \overset{\times d}\longrightarrow F^e_*R
$$
is $R$-linear. Precomposing with  $\lambda_c^e$ we have
$$
R \overset{\lambda^e_c}\longrightarrow F^e_*R \overset{\times d}\longrightarrow F^e_*R \,\,\,\,\, {\text{sending}} \,\, 1 \mapsto cd,
$$
which is $\lambda_{cd}^e$. Our hypothesis that $R$ is F-pure along $cd$ means that   there is some $e$ for which this composition is pure. So by Lemma \ref{properties of pure maps}(b), it follows also that 
$\lambda_c^e$ is pure. That is, $R$ is F-pure along $c$ (and since $R$ is commutative, along $d$).
 The second statement follows since F-purity along the product $c \times 1$ implies $R$ is  F-pure along $1$. So some iterate of Frobenius is a pure map, and so F-purity follows from Lemma
 \ref{properties of pure maps}(b).
{}

  (b) By (a) we see that $R$ is F-pure. In particular, the Frobenius map is pure and hence injective, so $R$ is reduced.

{}

  (c) Note $R$ is reduced by (b). Let $\alpha \in S^{-1}R$ be a non-zerodivisor. 
Because $R$ has finitely many minimal primes, a standard prime avoidance argument  shows that there exists a non-zerodivisor $c \in R$ and $s \in S$ such that $\alpha = c/s$  
(a minor modification of \cite[Proposition on Pg 57]{Hoc2}).
By hypothesis, $R$ is F-pure along $c$. Hence there exists $e > 0$ such that the map $\lambda_c^e: R \rightarrow F_*^{e}R$ is pure. Then the map
$$
\lambda_{c/1}^e: S^{-1}R \longrightarrow F_*^{e}(S^{-1}R)\,\,\,\,{\text{sending}}\,\,\,\, \hspace{2mm} 1 \mapsto c/1$$
is pure by \ref{properties of pure maps}(e) and the fact that $S^{-1}(F_*^{e}R) \cong F_*^{e}(S^{-1}R)$ as $S^{-1}R$-modules (the isomorphism $S^{-1}(F_*^{e}R) \cong F_*^{e}(S^{-1}R)$ is given by $r/s \mapsto r/s^{p^e}$). Now the $S^{-1}R$-linear map
$$\ell_{1/s}: S^{-1}R \rightarrow S^{-1}R\,\,\,\,{\text{sending}}\,\,\,\, \hspace{2mm} 1 \mapsto 1/s$$
is an isomorphism. Applying $F_*^{e}$, we see that
$$F_*^{e}(\ell_{1/s}): F_*^{e}(S^{-1}R) \rightarrow F_*^{e}(S^{-1}R)\,\,\,\,{\text{sending}}\,\,\,\, 1 \mapsto 1/s$$
is also an isomorphism of $S^{-1}R$-modules. In particular, $F_*^{e}(\ell_{1/s})$ is a pure map of $S^{-1}R$-modules. So purity of 
$$F_*^{e}(\ell_{1/s}) \circ \lambda_{c/1}^e$$ 
follows by \ref{properties of pure maps}(a). But $F_*^{e}(\ell_{1/s}) \circ \lambda_{c/1}^e$ is precisely the map
$$\lambda_{c/s}^e: S^{-1}R \rightarrow F_*^{e}(S^{-1}R)\,\,\,\,{\text{sending}}\,\,\,\, 1 \mapsto c/s.$$

{}

  (d) Let $c \in R$ be a non-zerodivisor. Then $\varphi(c)$ is a non-zero divisor in $T$ by hypothesis. Pick $e > 0$ such that the map $\lambda_{\varphi(c)}^e: T \rightarrow F_*^{e}T$ is a pure map of $T$-modules. By \ref{properties of pure maps}(f) and \ref{properties of pure maps}(a),
$$R \xrightarrow{\varphi} T \xrightarrow{\lambda_{\varphi(c)}^e} F_*^{e}T$$
is a pure map of $R$-modules. We have commutative diagram of $R$-linear maps
$$\begin{tikzcd}
R \arrow{r}{\varphi} \arrow{d}{\lambda_{c}^e}
& T  \arrow{d}{\lambda_{\varphi(c)}^e}\\
F_*^{e}R \arrow{r}{F_*^{e}(\varphi)} & F_*^{e}T
\end{tikzcd}$$

  The purity of $\lambda_c^e$ follows by \ref{properties of pure maps}(b). Note that if $\varphi$ is faithfully flat, then it is pure by \ref{properties of pure maps}(f) and maps non-zerodivisors to non-zerodivisors. 

{}

  (e) Let $R := R_1 \times \dots \times R_n$. Consider the multiplicative set
$$S : = R_1 \times \dots \times R_{i-1} \times \{1\} \times R_{i+1} \times \dots \times R_n.$$
Since
$S^{-1}R \cong R_i,$
it suffices to show that $S^{-1}R$ is  F-pure regular. So let $\alpha \in S^{-1}R$ be a non-zerodivisor. Note that 
we can select $u \in R$ and $s \in S$ such that $u$ is a non-zerodivisor and $\alpha = u/s$.
So we can now repeat the proof of (c) verbatim to see that  $S^{-1}R$ must be pure along $\alpha$.
\end{proof}

{}

\begin{remark}
\label{technical remark on purity} 
It is worth observing in Definition \ref{pure element and pure F-regularity}, that if the map  $\lambda_c^e$ is a pure map, then $\lambda_c^f$ is also a pure map for all $f \geq e$. Indeed, to see this note that it suffices to show that $\lambda_c^{e+1}$ is pure. We know $R$ is F-pure by \ref{basic properties of pure F-regularity}(a). So  Frobenius 
$$F: R \rightarrow F_*R$$
is a pure map of $R$-modules. By hypothesis, 
$$\lambda_c^e: R \rightarrow F_*^{e}{R}$$
is pure. Hence \ref{properties of pure maps}(d) tell us that
$$F_*(\lambda_c^e): F_*R \rightarrow F_*(F_*^eR)$$
is a pure map of $R$-modules. Hence the composition
$$R \xrightarrow{F} F_*R \xrightarrow{F_*(\lambda_c^e)} F_*(F_*^{e}R)\,\,\,\,{\text{sending}}\,\,\,\, 1 \mapsto c$$
is a pure map of $R$-modules by \ref{properties of pure maps}(a). But $F_*(F_*^{e}R)$ as an $R$-module is precisely $F_*^{e+1}R$. So 
$$\lambda_c^{e+1}: R \rightarrow F_*^{e+1}R.$$
is pure.
\end{remark}

\begin{example}\label{nonNoetherianFregular} The polynomial ring over $\mathbb F_p$ in infinitely many variables (localized at the obvious maximal ideal) is an example of a F-pure regular ring which is not Noetherian. 
\end{example}

{}

\subsection{Relationship of F-pure regularity to other singularities}
We show that our generalization of strong F-regularity continues to enjoy many important properties of the more restricted version.

\begin{theorem}
\label{Regular local rings are purely F-regular} (C.f. \cite[Theorem 3.1(c)]{HH1})
A regular local ring, not necessarily F-finite,  is F-pure regular. 
\end{theorem}

\begin{proof} Let $(R, m)$ be a regular local ring. By Krull's intersection theorem we know that
$$\bigcap_{e > 0} m^{[p^e]} = 0.$$
Since $R$ is a domain, the non-zerodivisors are precisely the non-zero elements of $R$. So let $c \in R$ be a non-zero element.  Choose  $e $ such that $c \notin m^{[p^e]}$. We show that the map
$$\lambda_c^e: R \rightarrow F_*^{e}R; \hspace{2mm} 1 \mapsto c$$
is pure. 

By Lemma \ref{properties of pure maps}, it suffices to check that for the injective hull $E$ of the residue field of $R$, the induced map
$$\lambda_c^e \otimes id_E: R \otimes_R E \rightarrow F_*^{e}R \otimes_R E$$
is injective, and for this, in turn,  we need only check that the socle generator  is not in the kernel.

 Recall that   $E$ is the direct limit of the injective maps
 $$
R/(x_1, \dots, x_n) \overset{x}\longrightarrow R/(x_1^2, \dots, x_n^2) \overset{x}\longrightarrow R/(x_1^3, \dots, x_n^3) \overset{x}\longrightarrow R/(x_1^4, \dots, x_n^4) \longrightarrow\cdots 
 $$
 where $x_1, \dots, x_n$ is a minimal set of generators for $ m,$ and the maps are given by multiplication by  $x = \Pi_{i=1}^d x_i.$ 
 So the module $F_*^eR \otimes_R E$ is the direct limit of the  maps
 $$
R/(x_1^{p^e}, \dots, x_n^{p^e}) \overset{x^{p^e}}\longrightarrow R/(x_1^{2p^e}, \dots, x_n^{2p^e}) \overset{x^{p^e}}\longrightarrow R/(x_1^{3p^e}, \dots, x_n^{3p^e}) \overset{x^{p^e}}\longrightarrow R/(x_1^{4p^e}, \dots, x_n^{4p^e}) \longrightarrow\cdots 
 $$
 which remains injective by the faithful flatness of $F_*^eR$.
  The induced map  $\lambda_c^e \otimes id_E: E \rightarrow F_*^e R \otimes E$ sends the socle (namely the image of $1$ in $R/m$) to the class of
   $c $ in $ R/m^{[p^e]}$, so it is non-zero provided $c\notin m^{[p^e]}.$ Thus for every non-zero $c$ in a regular local (Noetherian) ring, we have found an $e$, such that the map $\lambda_c^e$ is pure. So regular local rings are  F-pure regular.
\end{proof}

{}

\begin{proposition}
\label{pure F-regularity and F-regularity}
An F-pure regular   ring   is   normal, that is, it  is integrally closed in its total quotient ring.
\end{proposition}

\begin{proof}
 Take a fraction $r/s$ in the total quotient ring integral over $R$.  Then clearing denominators in an equation of integral dependence, we have 
$r \in \overline{(s)}$, the integral closure of the ideal $(s)$. This implies that there exists an $h$ such that $(r,s)^{n+h} = (s)^n(r, s)^h$ for all $n$ \cite[p64]{Mat}.  Setting $c=s^h$, this implies
$ cr^n \in (s)^n$ for all large  $n$. In particular, taking $n = p^e$, we see that class of  $r$ modulo $(s)$ 
is in the kernel of the map induced by tensoring the map
\begin{equation}\label{eq5}
R\rightarrow F^e_*R \,\,\,{\text{ sending}} \,\,\, 1\mapsto c
\end{equation}
 with the quotient module $R/(s)$. 
By purity of the map (\ref{eq5}), it follows that 
$r\in (s)$. 
We conclude that $r/s$ is in $R$ and that $R$ is normal.
\end{proof}

\subsection{Connections with  Tight Closure}
In his lecture notes on tight closure \cite{Hoc2}, Hochster suggests another way to generalize strong F-regularity to non-F-finite (but Noetherian) rings using tight closure.
We show here that  his  generalized strong F-regularity is the same as F-pure regularity for local Noetherian rings.

Although Hochster and Huneke introduced tight closure only in Noetherian rings, we can make the same definition in general for an arbitrary ring of prime characteristic $p$.
Let $N \hookrightarrow  M$ be $R$-modules. The {\textbf{ tight closure }} of $N$ in $M$ is an $R$-module $N^*_M$ containing $N$. By definition, an element $x\in M$ is in  $N^*_M$ if there exists $c\in R$, not in any minimal prime, such that for all sufficiently large $e$, the element $c\otimes x \in F^e_*R \otimes_R M$ belongs to the image of the module $F^e_*R \otimes_R N$ under the natural map 
 $ F^e_*R \otimes_RN \rightarrow  F^e_*R \otimes_R M$ induced by tensoring the inclusion $N \hookrightarrow  M$ with the $R$-module $F^e_R$.  We say that $N$ is tightly closed in $M$ if $N^*_M = N$.

\begin{definition}\label{HochsterF-regular}
Let $R$ be a Noetherian ring of prime characteristic $p$. We say that  $R$ is {\it strongly F-regular} in the sense of Hochster if, for any pair of $R$ modules  $N \hookrightarrow  M$, $N^*_M= N$.
\end{definition}

{} 
The next result compares  F-pure regularity with strongly F-regularity in the sense of Hochster:

\begin{proposition}\label{Hochster=strongF-regular}
Let $R$ be an arbitrary commutative ring of prime characteristic.  If $R$ is F-pure regular, then $N$ is tightly closed in $M$ for any pair of $R$ modules $N \subset M$.
The converse also holds if $R$ is Noetherian and local.
\end{proposition}

\begin{proof}
(i) Suppose $x \in N^*_M$.  Equivalently the class $\overline x$  of $x$ in $M/N$ is in $0^*_{M/N}$. So there exists  $c$ not in any minimal prime such that $c\otimes \overline x = 0 $ in $F^e_*R \otimes_R M/N$ for all large $e$. But this means that the map 
$$
R\rightarrow F^e_* R \,\,\,\,\,{\text{sending}}\,\,\, 1\mapsto c
$$ is not pure for any $e$, since
the naturally induced map 
$$
R \otimes M/N \rightarrow F^e_* R \otimes M/N 
$$
has 
$1\otimes \overline {x}$ in its kernel.

For the converse, let $c\in R$ be not in any minimal prime. We need to show that there exists some $e$ such that  the map 
$R\rightarrow F^e_*R$ sending $1$ to $c$ is pure. 
Let $E$ be the injective hull of the residue field of $R$.
 According to Lemma \ref{properties of pure maps}(i), it suffices to show that there exists an $e$ such that after tensoring $E$,  the induced map
 $$R\otimes  E\rightarrow F^e_*R \otimes E$$
is injective. But if not, then a generator $\eta$  for the socle of $E$ is in the kernel for every $e$, that is, for all $e$, $c\otimes \eta =0$ in $F^e_*R \otimes E$. In this case, $\eta \in 0^*_E$, contrary to our hypothesis that all modules are tightly closed.
\end{proof}

\begin{remark}
We do not know whether Proposition \ref{Hochster=strongF-regular} holds in the non-local case. Indeed, we do not know if F-pure regularity is a local property: if $R_m$ is F-pure regular for all maximal ideals $m$ of $R$, does it follow that $R$ is F-pure regular? If this were the case, then our argument above extends to arbitrary Noetherian rings.

\end{remark}

\begin{remark} A Noetherian ring of characteristic $p$ is {\textbf{weakly F-regular}} if $N$ is tightly closed in $M$ for any pair of {\it  Noetherian} $R$ modules $N \subset M.$ Clearly F-pure regular implies weakly F-regular. The converse is a long standing open question in the F-finite Noetherian case. For valuation rings, however, our arguments show that weak and pure F-regularity are equivalent (and both are equivalent to the valuation ring being Noetherian); See Corollary \ref{Linquan}.
\end{remark}

{}

\subsection{Elements along which F-purity fails}\label{splittingprime} We now observe an analog of the splitting prime of Aberbach and Enescu \cite{AbEn}; See also  \cite[4.7]{Tuck1}.

{}

\begin{proposition} 
\label{when the set of non pure elements is prime}
Let  $R$  be a ring of characteristic $p$, and consider the set
$$\mathcal{I} := \{c \in R: \, R {\text{ is not F-pure along }} \, c\}.$$ 
Then $\mathcal I$ is closed under multiplication by $R$, and $R - \mathcal I$ is multiplicatively closed. In particular, if $\mathcal I$ is closed under addition, then  $\mathcal I $ is a prime ideal (or the whole ring $R$). 
\end{proposition}

{}

\begin{proof}[Proof of Proposition \ref{when the set of non pure elements is prime}]
We first note that $\mathcal I$ is closed under multiplication by elements of $R$. Indeed, suppose that $c\in \mathcal I$ and $r \in R$. Then if $rc \notin \mathcal I$, we have that $R$ is F-pure along $rc$, but this implies $R$ is F-pure along $c$ by Proposition \ref{basic properties of pure F-regularity}(a), contrary to $c \in \mathcal I$.

We next show that the complement $ R \setminus \mathcal I$ is a multiplicatively closed set (if non-empty). To wit, take $c, d \notin \mathcal I$. Because  $R$ is F-pure along both $c$ and $d$, 
we have that there exist $e$ and $f$ such such the maps 
$$ R \overset{\lambda_c^e}\longrightarrow F_*^{e}R\,\,\,{\text{sending}} \,\,\, 1 \mapsto c, \,\,\,\,\,{\text{and}} \,\,\,  R \overset{\lambda_d^f}\longrightarrow F_*^{f}R
\,\,\,{\text{sending}}\,\,\, 1 \mapsto d$$
are both pure. Since purity is preserved by restriction of scalars (Lemma \ref{properties of pure maps}(d)), we also have that 
$$
F^e_*R \overset{F^e_*(\lambda_d^f)}\longrightarrow F^e_*F_*^{f}R  = F^{e+f}_*R
$$ is pure. Hence the composition 
$$R \overset{\lambda_c^e}\longrightarrow F_*^{e}R \overset{\lambda_d^f}\longrightarrow F^e_*F_*^{f}R \,\,\,\,{\text{sending}} \,\,\,\, 1 \mapsto c^{p^e}d$$
 is pure as well (Lemma \ref{properties of pure maps}(a)).  This means that $c^{p^e}d$ is not in $\mathcal I$, and since $\mathcal I$ is closed under multiplication, neither is $cd$. Note also that if $R \setminus \mathcal{I}$ is non-empty, then $1  \in R \setminus \mathcal I$ by Proposition \ref{basic properties of pure F-regularity}(a). Thus $R\setminus \mathcal I$ is a multiplicative set.
 
 Finally, if $\mathcal I$ is closed under addition (and $\mathcal I \neq R$), we conclude that $\mathcal I$ is a prime ideal since it is an ideal whose complement is a multiplicative set.
\end{proof}

{}
 
\begin{remark} If $R$ is a Noetherian local  domain, then  the set $\mathcal{I}$ of Proposition 
 \ref{when the set of non pure elements is prime} can be checked to be closed under addition (see, for example,  \cite[4.7]{Tuck1} for the F-finite case).  Likewise,  for valuation rings,  the set $\mathcal{I}$ is also an ideal: we construct it explicitly in the next section.
   However, for an arbitrary ring, 
 $\mathcal{I}$ can fail to be an ideal. For example, 
  under suitable  hypothesis, the set $\mathcal{I}$ is also the union of the centers of F-purity in the sense of Schwede, hence in this case, $\mathcal{I}$ is a finite union of ideals but not necessarily an ideal in the non-local case; see \cite{Schw1}. 
\end{remark}

{}

\subsection{F-pure regularity and Valuation Rings}
\label{pure F-regularity for valuations}

{}

  In this subsection we characterize valuation rings that are F-pure regular. The main result is:

\begin{theorem}
\label{purely F-regular valuation rings}
A valuation ring is F-pure regular if and only if it is Noetherian. Equivalently, a valuation ring is F-pure regular if and only if it is a field or a DVR.
 \end{theorem}
 
{}
 
  A key ingredient in the proof is the following theorem about the set of elements  along which $V$ fails to be F-pure  (C.f. Definition \ref{pure element and pure F-regularity}):

{}

\begin{theorem}\label{ideal}
\label{elements of a valuation ring that are pure} 
The set of elements $c$ along which a valuation ring $(V, m)$ fails to be F-pure is the prime ideal
$$\mathcal{Q} := \bigcap_{e > 0} m^{[p^e]}.$$
\end{theorem}

\begin{proof}
First, take any $c \in \mathcal{Q}$. We need to show that $V$ is not F-pure along $c$, that is, that the map
$$\lambda_c^e: V \rightarrow F_*^{e}V \,\,\,\,\,\,{\text{sending}}\,\,\, 1 \mapsto c$$
is not pure for any $e$. Because $c \in m^{[p^e]}$, we see that tensoring with 
$\kappa := V/m$ produces the zero map. So $\lambda_c^e$ is not  pure for any $e$, which means $V$ is not F-pure along $c$.

{}

  For the other inclusion, let $c \notin m^{[p^e]}$ for some $e > 0$. We claim that $\lambda_c^e: V \rightarrow F_*^{e}V$ is pure. Apply Lemma \ref{properties of pure maps}(g) to the set $\Sigma$ of finitely generated submodules of $F_*^eV$ which contain $c$. Note $\Sigma$ is a directed set under inclusion with a least element, namely the $V$-submodule of $F_*^eV$ generated by $c$, and $F_*^eV$ is the direct limit of the elements of $\Sigma$. It suffices to show that if $T \in \Sigma$, then 
$$\lambda_T: V \rightarrow T \,\,\,\,{\text{sending}}\,\,\,\, 1 \mapsto c$$
is pure. But $T$ is free since it is a finitely generated, torsion-free module over a valuation ring (Lemma \ref{finitely generated torsion free modules over valuation rings}). Since $c \notin m^{[p^e]}$, by the $V$ module structure on $T$, we get $c \notin mT$. By Nakayama's Lemma, we know $c$ is  part of a free basis for $T$. So $\lambda_T$ splits, and is pure in particular. 

{}

  Now that we know that the set of elements along which $R$ is not F-pure is an ideal, it follows that it is a  prime ideal from Proposition \ref{when the set of non pure elements is prime}. 
\end{proof}

{}

\begin{corollary}
\label{necessary and sufficient for purely F-regular} For a valuation ring  $(V, m) $ of characteristic $p$,  define  $\mathcal{Q} := \bigcap_{e > 0} m^{[p^e]}$.
Then the quotient $V/\mathcal{Q}$ is a F-pure regular valuation ring. Furthermore,  $V$ is  F-pure regular if and only if $\mathcal Q$ is zero.
\end{corollary}

\begin{proof}
The second statement follows immediately from Theorem \ref{elements of a valuation ring that are pure}. For the first, observe that $V/\mathcal{Q}$ is a domain since $\mathcal Q$ is prime. So ideals of $V/\mathcal{Q}$ inherit the total ordering under inclusion from $V$, and $V/\mathcal{Q}$ is a valuation ring whose maximal  $\overline m$ ideal satisfies
$ \bigcap_{e > 0} {\overline m}^{[p^e]} = 0.$ So $V/\mathcal{Q}$ is F-pure regular.
\end{proof}

{}\begin{corollary}\label{Linquan} For a valuation ring,  F-pure regularity is equivalent to all ideals (equivalently, the maximal ideal) being tightly closed. 
\end{corollary}

\begin{proof}  Proposition \ref{Hochster=strongF-regular} ensures that F-pure regularity implies all ideals are tightly closed. For the converse, note that if 
 there is some non-zero $c $ in $ \bigcap_{e>0} m^{[p^e]}$, then  $1\in m^*$. So for any proper ideal $m$, the condition that $m^*=m$ implies that  $ \bigcap_{e>0} m^{[p^e]} = 0$.
 In particular, if the maximal ideal of a valuation ring $V$ is tightly closed, then Corollary \ref{necessary and sufficient for purely F-regular} implies that $V$ is F-pure regular.
\end{proof}

\begin{proof}[Proof of Theorem \ref{purely F-regular valuation rings}]
First observe that if $V$ is a field or DVR, then it is F-pure regular. Indeed, every map of modules over a field is pure (since all vector space maps split). And a DVR is a one dimensional regular local ring, so it is F-pure regular by Theorem \ref{Regular local rings are purely F-regular}.

{}

 Conversely, we show that if $(V, m)$ is F-pure regular, its dimension is at most one. Suppose  $(V,m)$ admits a non-zero prime ideal $P \neq m$. Choose $x \in m \setminus P$, and  a non-zero element $c \in P.$  
The element $c$ {\it cannot}  divide  $x^n$ in $V$,  since  in that case  we would have $x^n \subset (c) \subset P$,  but $P$ is a prime ideal not containing $x$.
It follows from the definition of  a valuation ring, then, that $x^n$ divides $c$ for all $n$. This  means in particular that $c \in (x)^{[p^e]} \subset m^{[p^e]}$ for all $e$.  So $c \in \mathcal{Q}.$ According to Theorem \ref{elements of a valuation ring that are pure}, $R$ is not F-pure regular. 

{}

  It remains to show that an  F-pure regular valuation ring $V$ of dimension one is discrete.  Recall that the value group $\Gamma$ of  $V$ is (order isomorphic to) an additive subgroup of $\mathbb{R}$  \cite[Theorem 10.7]{Mat}.

{}

  We claim that   $\Gamma$ has a least positive element. To see this, let  $\eta$ be the greatest lower bound of all positive elements in $\Gamma$.  First observe that $\eta$ is strictly positive. Indeed, for 
fixed $c\in m$,  the sequence $\frac{v(c)}{p^e}$ consists of positive real numbers approaching zero as $e$ gets large. If $\Gamma$ contains elements of arbitrarily small positive values, then we could find  $x \in V$ such that 
$$0 < v(x) < \frac{v(c)}{p^e}.$$  But then 
 $0 < v(x^{p^e}) < v(c)$, which says that $c \in (x)^{[p^e]} \subset m^{[p^e]}$ for all $e$. This contradicts our assumption that  $V$ is  F-pure along $c$ 
(again, using Theorem \ref{elements of a valuation ring that are pure}).

{}

  Now that we know the greatest lower bound $\eta$ of $\Gamma$ is positive, it remains to show that   $\eta \in \Gamma$. Choose  $\epsilon$ such that $0 < \epsilon < \eta.$
If $\eta \notin \Gamma$, we know $\eta < v(y)$ for all $y \in m$. Since $\eta$ is the greatest lower bound, we can find  $y$ such that 
$$\eta < v(y) < \eta + \epsilon, $$ as well as 
 $x$ such that 
$$\eta < v(x) < v(y) < \eta + \epsilon.$$ 
Then 
$$0 < v(y/x) < \epsilon < \eta,$$
contradicting the fact that $\eta$ is a lower bound for $\Gamma$.  We conclude that  $\eta \in \Gamma$, and  that $\Gamma$ has a least positive element.

{}

   It is now easy to see using the Archimedean axiom for real numbers that the ordered subgroup $\Gamma$ of $\mathbb R$ is generated by its least positive element $\eta$.  In particular, $\Gamma$ is order isomorphic to $\mathbb{Z}$. We conclude that   $V$ is a DVR. 
\end{proof}

{}

\begin{remark}
\label{uniform F-compatiblity}
For a valuation ring $(V, m)$ of dimension $n \geq 1$, our results show that in general
$$\mathcal{Q} = \bigcap_{e \in \mathbb{N}} m^{[p^e]}$$
is a prime ideal of height at least $n-1$. It is easy to see that the situation where  $V/\mathcal{Q}$ is a DVR arises if and only if $m$ is principal, which in turn is equivalent to the  value group $\Gamma$ having a least positive element. For example, this is the case for the lex valuation in Example \ref{lex}.
It is not hard to check that $\mathcal Q$ is  uniformly F-compatible ideal in the sense of Schwede \cite{Schw1} (see also \cite[3.1]{Smith-Zhang} for further discussion of uniformly F-compatible ideals), generalizing of course to the non-Noetherian and non-F-finite setting. 
A general investigation of uniformly F-compatible ideals appears to be  fruitful, and is being undertaken by the first author. 
\end{remark}

{}

\subsection{Split F-regularity}
\label{split F-regularity for valuations}
  Of course, there is another obvious way{\footnote{This generalization is used for cluster algebras in  \cite{BeMuRaSm} for example.}} to adapt Hochster and Huneke's definition of strongly F-regular to arbitrary rings of prime characteristic $p$:

\begin{definition}
\label{split F-regularity}
A ring $R$ is \textbf{split F-regular} if for all non-zero divisors $c$, there exists $e$ such that the map $R \rightarrow F^e_*R$  sending $1$ to $c$ splits as a map of $R$-modules.
\end{definition}

{}

  Since split maps are pure,  a split F-regular ring is F-pure regular. Split F-regular rings are also clearly Frobenius split.  On the other hand, Example \ref{A discrete valuation ring which is not F-finite} shows that a discrete valuation ring need not be Frobenius split, so split F-regularity is strictly stronger than F-pure regularity.  In particular, not every regular local ring is split F-regular, so split F-regularity should not really be considered a class of "singularities" even for Noetherian rings.

{}

\begin{remark}
In Noetherian rings, split F-regularity is very close to F-pure regularity. For example, if  $R$ is  an F-pure regular Noetherian domain whose fraction field is F-finite, then
 the only obstruction to split F-regularity is the splitting of Frobenius.
 This is a consequence of Theorem \ref{F-finiteFSplit}, which tells us $R$ is F-finite if it is Frobenius split,  and Lemma \ref{splitvspure}, which tells us F-split and F-pure are the same in F-finite Noetherian rings.
\end{remark}

\begin{corollary}\label{splitDVR}
For a discrete valuation ring  $V$ whose fraction field is F-finite, the following are equivalent:
\begin{enumerate}
\item[(i)] $V$ is split F-regular;
\item[(ii)] $V$ is  Frobenius split;
\item[(iii)] $V$ is F-finite;
\item[(iv)] $V$ is free over $V^p$;
\item[(v)] $V$ is excellent.
\end{enumerate}
Moreover, if $K$ is a function field over an F-finite ground field $k$, and $V$ is a valuation of $K/k$, then (i)-(v) are equivalent to $V$ being a divisorial valuation ring.
\end{corollary}

\begin{proof}
All this has been proved already. Recall that a DVR is a regular local ring, so it is always F-pure regular and hence split F-regular if it is F-finite. Also, the final statement follows from Theorem \ref{F-finiteness equivalent to Abhyankar} because an Abyhankar valuation of rational rank one is necessarily divisorial, and a divisorial valuation of a functional field over an F-finite field is necessarily $F$-finite. 
\end{proof}

To summarize: A valuation ring is F-pure regular if and only if it is Noetherian,  and split F-regular (under the additional assumption that its fraction field is F-finite) if and only if it is excellent.
{}

{}

\section{Concluding Remarks}
  We have argued that for valuation rings, F-purity and  F-pure regularity (a version of strong F-regularity defined using pure maps instead of split maps)  are natural and robust 
 properties. We have also seen that the conditions of Frobenius splitting and split F-regularity are  more subtle, and that even regular rings can fail to satisfy these.
 
For Noetherian valuation rings in F-finite fields, we have seen that the Frobenius splitting property is equivalent to F-finiteness and also to excellence, but we do not know what happens in the non-Noetherian case: does there exist an example of a (necessarily non-Noetherian) Frobenius split valuation ring of an F-finite field that is \textit{not} F-finite? {By Corollary \ref{F-finiteness implies finitely generated value group}, a possible strategy could be to construct a Frobenius split valuation ring in a function field whose value group is infinitely generated.} For example, can one construct an F-split valuation in $\mathbb F_p(x, y)$ with value group $\mathbb Q$?
{On the other hand, perhaps Frobenius splitting is equivalent to F-finiteness (just as in the Noetherian case). One might then ask whether a generalized version of Theorem \ref{F-finite iff free} holds for arbitrary fields: is a valuation ring Frobenius split if and only if Frobenius is free?} 

   {We propose that F-pure regularity is a  more natural generalization of strong F-regularity to the non-F-finite case than  a suggested  generalization  of strong F-regularity using tight closure due to Hochster. We have seen that F-pure regularity implies Hochster's notion, and that they are equivalent for local Noetherian rings.     However,  we do not know whether F-pure regularity is a local notion: if $R_m$ is  F-pure regular for all maximal ideals, does it follow that $R$ is F-pure regular?   We expect this to be true but  the standard arguments are insufficient to prove it.    (In the Noetherian F-finite case, this is well known; C.f.  \cite[Theorem 5.5(a)]{HH4}. Furthermore, the answer is affirmative for excellent rings with F-finite total quotient rings by Proposition \ref{F-finiteExcellent}.)  If true, then F-pure regularity would be equivalent to all modules being tightly closed in the Noetherian case. More generally, 
might F-pure regularity be equivalent to the property that all modules are tightly closed even in the non-Noetherian case? Or even that all ideals are tightly closed? An affirmative answer to this last question would imply that strong and weak  F-regularity are equivalent.}
{}


\bibliographystyle{plain}
\bibliography{Bibliography}

\end{document}


\begin{abstract}
Theorem 5.1 is corrected in the paper \emph{Frobenius and valuations rings}, DOI: 10.2140/ant.2016.10.1057.
\end{abstract}
\maketitle



The characterization of F-finite valuation rings of function fields (Theorem {5.1})  is incorrect. The corrected characterization is
  \begin{theorem}
\label{F-finite valuation rings in function fields} Let $K$ be a finitely generated extension of an F-finite ground field $k$. 
Let $V$ be any {non-trivial} valuation ring of $K/k$. Then $V$ is F-finite if and only if $V$ is divisorial.
\end{theorem}

A proof of this corrected theorem follows a complete accounting of affected statements in the paper. Notation, Theorem and page numbers  are as in \cite{DatSmi}.

\section{List of Affected Statements}
The source of the error is {Proposition} 2.2.1 (pg. 1061), which  is misquoted from the original source, \cite{Bour}. Specifically, the last sentence  {in the statement of Proposition 2.2.1} should read:
 {\it Furthermore, equality holds {\bf if}  the integral closure of $R_{\nu}$ in L is a finitely generated $R_{\nu}$-module,}   not  {\bf if and only if.}
The statement {\it is} correct, however, under the additional assumption that 
  the valuation ring $R_{\nu}$ is   {\it Noetherian}  \cite[VI, $\mathsection 8.5$, Remark (1)]{Bour}.
  
  This error has consequences in  the following statements from the paper: 
  
    \noindent
{\sc{Theorem 4.3.1}}  should read: \,\, {\it Let $V$ be a valuation ring of an F-finite field $K$  of prime characteristic $p$. If $V$  is F-finite, 
then $$[\Gamma:p\Gamma][\kappa:\kappa^p] = [K:K^p],$$ where $\Gamma$ is the value group and $\kappa$ is the  residue field of $V$.}

The paper had incorrectly stated the converse, which {\it does} hold if the valuation ring is Noetherian {\cite[VI, $\mathsection 8.5$, Remark (1)]{Bour}}.  

\noindent 
{\sc Corollary 4.3.2} should read: \label{revised Corollary 4.3.2} {\it With hypotheses as in Thm 4.3.1, if $[K:K^p] = [\kappa:\kappa^p],$ then $V$ is F-finite.}
 See Section \ref{proofs} for the proof.
 The original statement holds as stated if $V$ is Noetherian.

    \noindent
{\sc Examples 4.4.1 and 4.4.2}  (pg. 1069): The computations here are correct but the conclusions are not.  Rather, Theorem \ref{F-finite valuation rings in function fields}  ensures these valuation rings are {\bf not} F-finite.

\noindent {\sc Example 4.5.1} (pg. 1070) is correct as stated, as it follows from  the correct implication of Theorem 4.3.1.

    \noindent
{\sc Proposition 4.6.1} (pg. 1071): This statement is correct. However, the proof of  (iii)  is not,  because it cited the incorrect implication of Proposition 2.2.1. The corrected proof is
in Section \ref{proofcorrection} below.

      \noindent
  {\sc Theorem {5.1}} (pg. 1072):  The equivalence of (ii) and (iii) is correct, and both imply (i). Under the additional assumption that the valuation is discrete, also (i) implies  (ii) and (iii).
       The proof of  Theorem {5.1} in the paper  proves:
    
    \begin{theorem}\label{NumCritAb}
Let $K$ be a finitely generated field extension of an F-finite ground field $k$ of characteristic $p$.
The following are equivalent for a valuation $v$ on $K/k$:
\begin{enumerate}
\item[(i)] The valuation $v$ is Abhyankar.
\item[(ii)] $[\Gamma:p\Gamma][\kappa:\kappa^p]  = [K:K^p]$
\end{enumerate}
where  $\Gamma$ is the value group, and $\kappa$ is the residue field of $v$.
    \end{theorem}
    
    \noindent
    {\sc Corollary {5.2}} (pg. 1072) is correct as stated; it is a consequence of  Theorem  \ref{F-finite valuation rings in function fields}.
    
      \noindent
    {\sc Corollary {6.6.3}} (pg. 1086) is correct as stated; the proof should invoke Theorem  \ref{F-finite valuation rings in function fields} instead of Theorem 5.1.

    The remaining statements in Section 5 {(Remark 5.3, Proposition 5.4, Lemma 5.5, Proposition 5.6, Lemma 5.7, Lemma 5.8)}  and Section 6 are correct, because they do not rely on Lemma 2.2.1 or its consequences. 
    
\section{Proof  of Theorem \ref{F-finite valuation rings in function fields}}\label{proofs}

  \begin{lemma}
\label{Frobenius power of max ideal equals it} Let  $(V, \mathfrak m)$ be a valuation ring  of prime characteristic $p$. If $\mathfrak m$ is not principal, then
$$\mathfrak m = \mathfrak{m}^{[p]}.$$
\end{lemma}

\begin{proof} 
Take any $x \in \mathfrak m$. It suffices to show that there exists $y \in \mathfrak{m}$ such that 
$y^p|x.$   Note that since $\mathfrak m$ is not principal, the value group
$\Gamma$ of the corresponding valuation $v$ {\it does not have a smallest element} $ > 0$. This means we can  choose $\beta_0 \in \Gamma$ such that
$0 < \beta_0 < v(x)$, and then also choose
$\beta_1 \in \Gamma$ such that
$$0 < \beta_1 < \operatorname{min}\{\beta_0, v(x) - \beta_0\}.$$
Let $y_1 \in \mathfrak m$ such that $v(y_1) = \beta_1$. Then
$$v(y^2_1) = 2\beta_1 < \beta_0 + (v(x) - \beta_0) = v(x).$$
Thus, $y_1^2|x$. Repeating this argument inductively, we can find $y_n\in \mathfrak m$ such that $y_n^2|y_{n-1}$ for all $n\in \mathbb N$.  For $n > \log_2p$, we have  
 $y_n^{2^n}$, and hence $y_n^p$, divides $x$.  
\end{proof}

\begin{lemma}
\label{dimension of V/m^[p]}
Let  $(V, \mathfrak m, \kappa)$ be  valuation ring of characteristic $p$.  Then the dimension of $V/\mathfrak m^{[p]}$ over $\kappa^p$ is
\begin{enumerate} 
\item[(a)] $[\kappa:\kappa^p]$ if $\mathfrak m$ is not finitely generated.
\item[(b)]  $p[\kappa:\kappa^p]$ if $\mathfrak m$ is finitely generated.
\end{enumerate}
\end{lemma}

\begin{proof}
Consider the short exact sequence {of $\kappa^p$-vector spaces}
\begin{equation}
\label{ses}
0 \rightarrow \mathfrak{m}/\mathfrak{m}^{[p]}\rightarrow V/\mathfrak{m}^{[p]} \rightarrow \kappa \rightarrow 0.
\end{equation}
If $\mathfrak m$ is not finitely generated, then  Lemma \ref{Frobenius power of max ideal equals it} implies that  $\mathfrak{m}/\mathfrak{m}^{[p]} = 0$, and  (a) follows.
 Otherwise, $\mathfrak m$ is principal, so $\mathfrak m^{[p]} = \mathfrak m^p$ and we have a
filtration
$$\mathfrak{m} \supsetneq \mathfrak{m}^2 \supsetneq \dots \supsetneq \mathfrak{m}^{p-1} \supsetneq \mathfrak{m}^{[p]} = \mathfrak{m}^p.$$
Since $\mathfrak{m}^{i}/\mathfrak{m}^{i+1} \cong \kappa$, we see that
$$\operatorname{dim}_{\kappa^p}(\mathfrak{m}/\mathfrak{m}^{[p]}) = (p-1)[\kappa:\kappa^p].$$  From the short exact sequence (\ref{ses}),
$\dim_{\kappa^p} (V/\mathfrak m^{[p]}) = p[\kappa:\kappa^p]$, proving
 (b).
\end{proof}

  \begin{proof}[Proof of Theorem \ref{F-finite valuation rings in function fields}] If $V$ is divisorial, it is a localization of a finitely generated algebra over the $F$-finite ground field $k$, 
  hence it is F-finite.

For the converse, let $\mathfrak m$ be the maximal ideal and $\kappa$ the residue field of $V$. Since $V$ is F-finite,
 by (the corrected) Theorem 4.3.1
\begin{equation}
\label{ramreseq}
 [\Gamma:p\Gamma] 
 [\kappa:\kappa^p]
 = [K:K^p],
\end{equation}
where 
$\Gamma$  is the value group of $V$.
Thus,  the valuation is  Abhyankar  {by Theorem \ref{NumCritAb} above}.  To show it is divisorial, we need  to show it is discrete.

Because  the value group of an Abhyankar valuation is finitely generated, 
$\Gamma \cong \mathbb{Z}^{\oplus s}$, for some integer $s \geq 1$. {Hence it suffices} to show that $s=1$.

Since $V$ is F-finite, we know 
$V$ is   free  over $V^p$  of rank $[K:K^p]$.  Tensoring  with the residue field $\kappa^p$  of $V^p$, we have that 
$V/\mathfrak{m}^{[p]}$
is also a free $\kappa^p$-module of rank $[K:K^p].$
Thus from (\ref{ramreseq}), we have
\begin{equation}
\dim_{\kappa^p}V/\mathfrak{m}^{[p]} = [K:K^p] =  [\Gamma:p\Gamma]   [\kappa:\kappa^p] =   |\mathbb Z^{\oplus s}/p\mathbb Z^{\oplus s}|[\kappa:\kappa^p]  =
   p^s  [\kappa:\kappa^p].
\end{equation}
But now since $p^s \neq 1$,  Lemma \ref{dimension of V/m^[p]} forces $s=1$, and the proof is complete.
\end{proof}

\vspace{2mm}

\begin{remark}  Theorem \ref{F-finite valuation rings in function fields} does not hold without {\it some} assumption on the field $K$.  For instance, if $K$ is {perfect},  Frobenius will be an isomorphism  (hence {a finite map}) for {\it any} valuation ring. The proof of Theorem \ref{F-finite valuation rings in function fields}  does show 
 that an F-finite valuation ring with finitely generated value group  must be Noetherian (hence discrete), without any restriction on its fraction field.  
 {Furthermore, the proof also shows that  a valuation ring  cannot be F-finite if its  value group $\Gamma$ satisfies $[\Gamma: p\Gamma] > p$.}
  \end{remark}

\begin{proof}[Proof of 
{\sc Revised Corollary 4.3.2:}] 
 
In general,  $[\Gamma: p\Gamma][\kappa:\kappa^p] \leq [K:K^p]$  by  \cite[VI, $\mathsection 8.1$, Lemma 2]{Bour}. Hence our hypothesis forces 
$[\Gamma: p\Gamma] = 1$. Thus $\Gamma = p\Gamma$, so that $\Gamma$ can not have a smallest positive element, which means that the maximal ideal $\frak m$ of $V$ is not finitely generated.
Lemma \ref{dimension of V/m^[p]}(a) now ensures  $\dim_{\kappa^p}(V/\mathfrak{m}^{[p]})= [\kappa:\kappa^p] = [K:K^p]$. Then $V$ is F-finite by \cite[VI, $\mathsection 8.5$, Theorem 2(c)]{Bour}.
\end{proof}

\section{Proof of  {\bf Proposition 4.6.1(iii):}}\label{proofcorrection}

We recall  {\bf Proposition 4.6.1(iii): }  {Let $K \hookrightarrow L$ be a finite extension of F-finite fields of characteristic $p$. 
Let $w$ be a valuation on $L$ and $v$ its restriction to $K$. Then the valuation ring of $v$ is F-finite if and only if the valuation ring of $w$ is F-finite.}

\begin{lemma}
\label{value groups with smallest positive element under field extensions}
With notation as in Proposition 4.6.1(iii), the maximal ideal  of the  valuation ring of $v$ is finitely generated if and only if the maximal ideal of the valuation ring of $w$ is finitely generated.
\end{lemma}

\begin{proof} For ideals in a valuation ring,  finite generation is the same being principal.   
Principality of the maximal ideal is equivalent to the value group having a smallest element $> 0$.
 Thus, it suffices to show that  the value group $\Gamma_v$ of $v$  has this property  if and only if $\Gamma_w$ does.
 
 Assume $\Gamma_w$ has a smallest element $g >0.$   We claim that for each $t \in \mathbb N$, the only positive elements of $\Gamma_w$ less than $tg$ are $g, 2g, \dots, (t-1)g$.
Indeed, suppose $0 < h < tg$. Since $g$ is smallest, $g \leq h < tg$, whence $0 \leq h-g < (t-1)g$. So by induction,  $g-h = ig$ for some $i \in \{0, 1, \dots, t-2\}$, and hence $h$ is among $g, 2g, \dots, (t-1)g$.
 
 Now, because $[\Gamma_w:\Gamma_v] \leq [L:K] < \infty$ by {\cite[VI, $\mathsection 8.1$, Lemma 2]{Bour}},     every element of $\Gamma_w/\Gamma_v$ is torsion.  Let 
 $n$ be the smallest positive integer  such that $ng \in \Gamma_v$. We claim that $ng$ is  the smallest positive element of $\Gamma_v$.
 Indeed, the only positive elements smaller {than $ng$} in $\Gamma_w$ are $g, 2g, \dots, (n-1)g$, and none of these are in $\Gamma_v$ {by our choice of $n$}.

Conversely, if $\Gamma_v$ has a smallest element $h > 0$, then the set 
$$S := \{g \in \Gamma_w: 0 < g < h\}$$ 
is finite because for distinct $g_1, g_2$ in this set, their classes in $\Gamma_w/\Gamma_v$ are also distinct, while $\Gamma_w/\Gamma_v$ is a finite group. Then the smallest positive element of $\Gamma_w$ is the smallest element of $S$, or $h$ if $S$ is empty. 
\end{proof}

\begin{proof}[Proof of Proposition 4.6.1(iii)] A  \emph{necessary and sufficient} condition for the F-finiteness of 
a  valuation ring $(V, \mathfrak m, \kappa)$ with  F-finite fraction field $K$ is, by \cite[VI, $\mathsection 8.5$, Theorem 2(c)]{Bour},  that 
\begin{equation}
\label{numerical criterion for F-finiteness}
\operatorname{dim}_{\kappa^p}(V/\mathfrak{m}^{[p]}) = [K:K^p].
\end{equation}
 Lemma \ref{dimension of V/m^[p]} gives a formula for $\operatorname{dim}_{\kappa^p}(V/\mathfrak{m}^{[p]})$ {in terms of $[\kappa: \kappa^p]$}
 that depends on whether {the maximal ideal} is finitely generated, which is the same for $v$ and $w$
  by Lemma \ref{value groups with smallest positive element under field extensions}. 
Proposition 4.6.1 (i) and (ii)  tells us that both $[K:K^p] = [L:L^p]$ and  $ [\kappa_v:\kappa_v^p] = [\kappa_w:\kappa_w^p].$ 
Thus 
Lemma  \ref{dimension of V/m^[p]} and Equation \ref{numerical criterion for F-finiteness}
guarantee that the valuation ring of $v$ is F-finite if and only if the valuation ring of $w$ is F-finite.
\end{proof}

\bibliographystyle{plain}
\bibliography{Erratabibliography}